\documentclass[11pt]{article}

\usepackage{amsmath,amsthm,amssymb,xypic}
\usepackage{hyperref}
\usepackage{graphicx}

\theoremstyle{plain}
    \newtheorem{theorem}{Theorem}[section]
    \newtheorem{lemma}[theorem]{Lemma}
    \newtheorem{corollary}[theorem]{Corollary}
    \newtheorem{proposition}[theorem]{Proposition}

 \theoremstyle{definition}
    \newtheorem{definition}[theorem]{Definition}

    \newtheorem{remark}[theorem]{Remark}

\theoremstyle{remark}

\numberwithin{equation}{section}

 \DeclareMathOperator{\Tr}{Tr}
 \DeclareMathOperator{\tr}{tr}

\DeclareMathOperator{\im}{im}
\DeclareMathOperator{\spec}{spec}
\DeclareMathOperator{\AS}{AS}

\DeclareMathOperator{\APS}{APS}

\DeclareMathOperator{\Id}{Id}

\DeclareMathOperator{\erfc}{erfc}

\DeclareMathOperator{\ind}{index}

\DeclareMathOperator{\End}{End}
\DeclareMathOperator{\Hom}{Hom}
\DeclareMathOperator{\sgn}{sgn}

\DeclareMathOperator{\Spin}{Spin}

         \DeclareMathOperator{\supp}{supp}

\DeclareMathOperator{\vol}{vol}
\begin{document}


\newcommand{\myemph}{\emph}

\newcommand{\Spinc}{\Spin^c}

    \newcommand{\R}{\mathbb{R}}
    \newcommand{\C}{\mathbb{C}}
    \newcommand{\N}{\mathbb{N}}
    \newcommand{\Z}{\mathbb{Z}}
    \newcommand{\Q}{\mathbb{Q}}
    \newcommand{\bT}{\mathbb{T}}
    \newcommand{\bP}{\mathbb{P}}

\newcommand{\g}{\mathfrak{g}}
\newcommand{\h}{\mathfrak{h}}
\newcommand{\p}{\mathfrak{p}}
\newcommand{\kg}{\mathfrak{g}}
\newcommand{\kt}{\mathfrak{t}}
\newcommand{\ka}{\mathfrak{a}}
\newcommand{\XX}{\mathfrak{X}}
\newcommand{\kh}{\mathfrak{h}}
\newcommand{\kp}{\mathfrak{p}}
\newcommand{\kk}{\mathfrak{k}}

\newcommand{\cA}{\mathcal{A}}
\newcommand{\cE}{\mathcal{E}}
\newcommand{\calL}{\mathcal{L}}
\newcommand{\calH}{\mathcal{H}}
\newcommand{\cO}{\mathcal{O}}
\newcommand{\cB}{\mathcal{B}}
\newcommand{\cK}{\mathcal{K}}
\newcommand{\cP}{\mathcal{P}}
\newcommand{\cN}{\mathcal{N}}
\newcommand{\calD}{\mathcal{D}}
\newcommand{\cC}{\mathcal{C}}
\newcommand{\calS}{\mathcal{S}}
\newcommand{\cM}{\mathcal{M}}
\newcommand{\cU}{\mathcal{U}}

\newcommand{\cCM}{\cC}
\newcommand{\PM}{P}
\newcommand{\DM}{D}
\newcommand{\LM}{L}
\newcommand{\vM}{v}

\newcommand{\sumGam}{\textstyle{\sum_{\Gamma}} }

\newcommand{\sigDg}{\sigma^D_g}

\newcommand{\Bigwedge}{\textstyle{\bigwedge}}

\newcommand{\ii}{\sqrt{-1}}

\newcommand{\Ubar}{\overline{U}}

\newcommand{\beq}[1]{\begin{equation} \label{#1}}
\newcommand{\eeq}{\end{equation}}

\newcommand{\Todo}{\textbf{To do}}

\newcommand{\mattwo}[4]{
\left( \begin{array}{cc}
#1 & #2 \\ #3 & #4
\end{array}
\right)
}

\newenvironment{proofof}[1]
{\noindent \emph{Proof of #1.}}{\hfill $\square$}

\title{An absolute version of the Gromov--Lawson relative index theorem}

\author{Peter Hochs\footnote{Institute for Mathematics, Astrophysics and Particle Physics, Radboud University, \texttt{p.hochs@math.ru.nl}}
{}
and Hang Wang\footnote{East China Normal University, \texttt{wanghang@math.ecnu.edu.cn}}}

\date{\today}

\maketitle

\begin{abstract}
A Dirac operator on a complete manifold is Fredholm if it is invertible outside a compact set. Assuming a compact group to act on all relevant structure, and the manifold to have a warped product structure outside such a compact set, 
we express the equivariant index of such a Dirac operator as an Atiyah--Segal--Singer type contribution from inside this compact set, and a contribution from outside this set. Consequences include equivariant versions of the relative index theorem of Gromov and Lawson, in the case of manifolds with warped product structures at infinity, and the Atiyah--Patodi--Singer index theorem.
\end{abstract}

\tableofcontents

\section{Introduction}

Let $M$ be a complete Riemannian manifold, and $S \to M$ a Hermitian vector bundle equipped with a Clifford action $c \colon TM \to \End(S)$. For a Hermitian connection $\nabla$ on $S$ that is compatible with $c$, we have the Dirac operator $D = c \circ \nabla$ on $\Gamma^{\infty}(S)$. This operator is Fredholm as a self-adjoint, unbounded operator on $L^2(S)$, under the condition that there are a compact subset $Z \subset M$ and a $b>0$ such that for all $s \in \Gamma_c^{\infty}(S)$ supported outside $Z$,
\beq{eq invtble at infty}
\|Ds\|_{L^2} \geq b \|s\|_{L^2}.
\eeq
See  \cite{Anghel93b, Gromov83}.

A natural index theorem for such operators is Gromov and Lawson's relative index theorem on  page 304 of \cite{Gromov83}: an expression for the difference of the indices of two such operators $D$, on possibly different manifolds, that agree outside their respective sets $Z$, in terms of Atiyah--Singer type integrals on those compact sets.

Another type of index problem that can be put into this context is Atiyah--Patodi--Singer (APS) type index theorems on compact manifolds with boundary. Then one attaches a cylinder to this boundry, to form a noncompact, but complete manifold without boundary, and studies the corresponding index problem there. This was already used in \cite{APS1}, but see also for example \cite{HWW}.

A widely studied class of index problems for Dirac operators on noncompact manifolds that does not fall into our context is the study of Callias-type operators \cite{Anghel89, Anghel93,  Bott78, Bunke95, Callias78, Kucerovsky01}. These have the form $D+\Phi$, for a vector bundle endmorphism $\Phi$ that makes the combined operator $D+\Phi$ invertible at infinity in the sense of \eqref{eq invtble at infty}, in cases where $D$ itself may not have this property. The operator $D+\Phi$ is not necessarily of the form $c \circ \nabla$, and theferore falls outside our current scope.

Our goal in this paper is to express the index of a Dirac operator that is invertible outside a compact set $Z$, with smooth boundary,  as a contribution of an Atiyah--Singer type contribution from $Z$ and a contribution from outside $Z$. 
We assume that there is a $G$-equivariant isometry $M \setminus Z \cong N \times (0,\infty)$, for a $G$-invariant, compact hypersurface $N \subset M$, with a Riemannian metric of the form
\beq{eq metric intro}
 e^{2\varphi}(g_N + dx^2),
\eeq
for a function $\varphi$ on $(0, \infty)$. 
Furthermore, we consider an action by a compact group $G$ on $M$, preserving all relevant structure, and consider the $G$-equivariant index $\ind_G(D)$ evaluated at a group element $g \in G$.  The index formula, Theorem \ref{thm index general}, then has the form
\beq{eq index intro}
\ind_G(D)(g) = \int_{Z^g}  \AS_g(D)+
\lim_{N \to \partial Z}
\int_0^{\infty}
\int_{N} \tr(g \lambda_s^P(g^{-1}n,n))\, dn
\, ds.
\eeq
Here
\begin{itemize}
\item $Z^g$ is the fixed-point set of $g$ in $Z$;
 \item
  $\AS_g(D)$ is the Atiyah--Segal--Singer integrand \cite{ASIII, BGV} associated to $D$;
\item the limit $\lim_{N \to \partial Z}$ means taking compact, $G$-invariant hypersurfaces $N \subset M \setminus Z$ closer and closer to $\partial Z$, in a way made precise in Corollary \ref{cor index general};
   \item  $ \lambda_s^P$ is the smooth kernel of the  operator  $e_P^{-s D_C^- D_C^+}D_C^-$, where $e_P^{-s D_C^- D_C^+}$  is the heat operator for the restriction $D_C$ of $D$ to $C = M \setminus Z$ with APS-type boundary conditions at $\partial Z$.
\end{itemize}

The proof of \eqref{eq index intro} includes an argument that the  second term on the right hand side converges. Besides the invertibility at infinity of $D$ in the sense of \eqref{eq invtble at infty}, we make some functional-analytic assumptions on the restriction of $D$ to  $M \setminus Z$ with APS-type boundary conditions. These boundary conditions are important in the follow-up paper \cite{HW21b}. Most of the proof of \eqref{eq index intro} is about splitting up an expression for the left hand side of \eqref{eq index intro} into contributions from inside and outside $Z$, incorporating the APS boundary conditions in a suitable way.

A direct consequence of \eqref{eq index intro} is an equivariant version of the Gromov--Lawson relative index theorem, for manifolds with ends of the type considered: if two operators agree outside compact sets, then the second terms on the right hand side of \eqref{eq index intro} cancel when the difference of their indices is taken. See Corollary \ref{cor GL}.

Another consequence of \eqref{eq index intro} is Donnelly's equivariant version of the APS index theorem, see Corollary \ref{cor APS}. In that setting, we compute the second term on the right hand side of \eqref{eq index intro} using an explicit expression for the heat operator $e_P^{-s D_C^- D_C^+}$ from \cite{APS1}, and obtain (an equivariant version of) the $\eta$-invariant of a Dirac operator on $\partial Z$. Here the APS boundary conditions at $\partial Z$ are important.

The computation in the APS setting raises the question when  the contribution from infinity in \eqref{eq index intro} can be computed explicitly, based on knowledge of the behaviour of $M$ and $D$ outside $Z$.
In  \cite{HW21b}, we impose conditions on the function $\varphi$ in \eqref{eq metric intro}, which are satisfied for the cylindrical ends used in the APS case, and hyperbolic cusps. Then we give more explicit expressions for the contribution from infinity in \eqref{eq index intro}.


A result analogous to \eqref{eq index intro} is  Theorem 1.1 in  \cite{CLWY20}. In that result, Dirichlet boundary conditions are used instead of the APS boundary conditions that we need for our purposes.
See also Remarks \ref{rem bdry cond} and \ref{rem APS bdry cond} for the relevance of these boundary conditions.

\subsection*{Acknowledgements}

PH is supported by the Australian Research Council, through Discovery Project DP200100729.
HW is supported by NSFC-11801178 and Shanghai Rising-Star Program 19QA1403200.

\section{The main result}

Throughout this paper, $M$ is a complete Riemannian manifold, and $S  = S^+ \oplus S^-\to M$ a $\Z/2$-graded Hermitian vector bundle.
We denote the Riemannian density on $M$ by $dm$. We also assume that a compact Lie group $G$ acts smoothly and isometrically on $M$, that $S$ is a $G$-equivariant vector bundle and that the action on $G$ preserves the metric and grading on $S$. We fix, once and for all, an element $g \in G$.

\subsection{Dirac operators that are invertible at infinity}\label{sec dirac ops}

We suppose that $S$ is a $G$-equivariant Clifford module, which means that there is a $G$-equivariant vector bundle homomorphism $c\colon TM \to \End(S)$, with values in the odd-graded endomorphisms, such that for all $v \in TM$,
\[
c(v)^2 = -\|v\|^2 \Id_S.
\]
Let $\nabla$ be a $G$-invariant, Hermitian connection on $S$ that preserves the grading. Suppose that for all vector fields $v$ and $w$ on $M$,
\[
[\nabla_v, c(w)] = c(\nabla^{TM}_v w),
\]
where $\nabla^{TM}$ is the Levi--Civita connection. Consider the Dirac operator
\beq{eq Dirac}
D\colon \Gamma^{\infty}(S) \xrightarrow{\nabla} \Gamma^{\infty}(S \otimes T^*M) \cong
 \Gamma^{\infty}(S \otimes TM) \xrightarrow{c}  \Gamma^{\infty}(S).
\eeq
It is odd with respect to the grading on $S$; we denote its restrictions to even- and odd-graded sections by $D^{\pm}$, respectively.

From now on, we assume that $D$ is \emph{invertible at infinity}, in the sense that there are a $G$-invariant compact subset $Z\subset M$ and a constant $b>0$ such that for all $s \in \Gamma^{\infty}_c(S)$ supported in $C:= M \setminus Z$,
\beq{eq D inv infty}
\|Ds\|_{L^2} \geq b \|s\|_{L^2}.
\eeq
For $k\geq 1$, let $W^k_D(S)$ be the completion of $\Gamma^{\infty}_c(S)$ in the inner product
\[
(s_1, s_2)_{W^k_D} := \sum_{j=0}^k(D^js_1, D^j s_2)_{L^2}.
\]
Because $D$ is invertible at infinity, it is Fredholm as an operator
\[
D\colon W^1_D(S) \to L^2(S).
\]
See Theorem 2.1 in \cite{Anghel93b} or Theorem 3.2 in \cite{Gromov83}.


There is no loss of generality in assuming that $N:= \partial Z$ is a smooth submanifold of $M$. We make this assumption from now on. We assume that 
there is a $G$-equivariant isometry
\beq{eq isom U}
C \cup N \cong N \times [0,\infty),
\eeq
mapping $N$ to $N \times \{0\}$, and
where the Riemannian metric on $N \times (0,\infty)$ is of the form
\beq{eq metric U}
 e^{2\varphi}(g_N + dx^2),
 \eeq
 where
\begin{itemize}
\item $\varphi \in C^{\infty}(0,\infty)$;
\item $g_N$ is the restriction of the Riemannian metric to $N$;
\item $x$ is the coordinate in $(0,\infty)$.
\end{itemize}

Furthermore, we assume that, with respect to the identification \eqref{eq isom U}, we have a $G$-equivariant isomorphism of vector bundles
\[
S|_U \cong S|_N \times (0,\infty) \to N \times (0,\infty),
\]
and
\beq{eq D on U}
D|_U = \sigma \left(f_1 \frac{\partial}{\partial x} + f_2 D_N + f_3 \right),
\eeq
where
\begin{itemize}
\item $\sigma \in \End(S|_N)^G$ is fibre-wise unitary, satisfies $\sigma^2 = 1$, and interchanges $S^+|_N$ and $S^-|_N$;
\item $f_1, f_2, f_3 \in C^{\infty}(0,\infty)$;
\item $D_N$ is a $G$-equivariant Dirac operator on $S|_N$ that preserves the grading.
\end{itemize}
We write $D_N^{\pm}$ for the restriction of $D_N$ to an operator on $\Gamma^{\infty}(S^{\pm}|_N)$.

\subsection{The index theorem} \label{sec index thm general}

Because $N$ is compact, $D_N^+$ has discrete spectrum. For $\lambda \in \spec(D_N^+)$, let $L^2(S^+|_N)_{\lambda}$ be the corresponding eigenspace. Let $J \subset \spec(D_N^+)$. Consider the orthogonal projection
\[
P^+\colon L^2(S^+|_N) \to \bigoplus_{\lambda \in J} L^2(S^+|_N)_{\lambda}.
\]
We will also use the projection
\[
P^- := \sigma_+ P^+ \sigma_+^{-1} \colon L^2(S^-|_N) \to  \sigma_+\left(\bigoplus_{\lambda \in J} L^2(S^+|_N)_{\lambda}\right),
\]
where $\sigma_+:= \sigma|_{(S^+|_N)}$.
Note that $P^-$ is not necessarily a spectral projection\footnote{In the situation of \cite{HW21b}, we have $D_N^- = -\sigma_+ D_N^+ \sigma_+^{-1}$, so $P^-$ is in fact the spectral projection for $D_N^-$ onto the eigenspaces corresponding to $-J$.} for $D_N^-$. We combine these to an orthogonal projection
\beq{eq def P}
P := P^+ \oplus P^-\colon L^2(S|_N) \to \bigoplus_{\lambda \in J} L^2(S^+|_N)_{\lambda} \oplus
\sigma_+\left(\bigoplus_{\lambda \in J} L^2(S^+|_N)_{\lambda}\right).
\eeq
We
 will omit the superscripts $\pm$ from $P^{\pm}$ from now on.

Consider  the vector bundle 
\[
S_C := S|_C \to C = M\setminus Z.
\]
For $k \in \N$, consider the Sobolev space
\[
W^k_D(S_C) :=\{s|_C; s \in W^k_D(S)\}.
%
\]
We denote the subspaces of even- and odd-graded sections by $W^k_D(S_C^{\pm})$, respectively.
We denote the restrictions of $D$ to operators from $W^{k+1}_D(S_C)$ to $W^k_D(S_C)$ by $D_C$.

Consider the spaces
\[
\begin{split}
W^1_D(S_C^+; P) &:= \{s \in W^1_D(S_C^+); P(s|_N) = 0\};\\
W^1_D(S_C^-;1- P) &:= \{s \in W^1_D(S_C^-); (1-P)(s|_N) = 0\};\\
W^2_D(S_C^+; P) &:= \{s \in W^2_D(S_C^+); P(s|_N) = 0,  (1-P)(D^+_Cs|_N) = 0\};\\
W^2_D(S_C^-; 1-P) &:= \{s \in W^2_D(S_C^-); (1-P)(s|_N) = 0,  P(D^-_Cs|_N) = 0\}.
\end{split}
\]
Here we use the fact that there are well-defined, continuous restriction/extension maps $W^1_D(S_C) \to L^2(S|_N)$.
We assume that
\begin{itemize}
\item the operators
\beq{eq DC+}
D^+_C\colon W^1_D(S_C^+; P) \to L^2(S_C^-)
\eeq
and
\beq{eq DC-}
D^-_C\colon W^1_D(S_C^-; 1-P) \to L^2(S_C^+)
\eeq
are invertible; and
\item the operators
\beq{eq DC-DC+}
D^-_C D_C^+\colon W^2_D(S_C^+; P) \to L^2(S_C^+)
\eeq
and
\beq{eq DC+DC-}
D^+_C D_C^-\colon W^2_D(S_C^-; 1-P) \to L^2(S_C^-)
\eeq
 are self-adjoint.
\end{itemize}
\begin{remark}\label{rem bdry cond}
These conditions on the operators \eqref{eq DC+}--\eqref{eq DC+DC-} depend on the set $J \subset \spec(D_N^+)$ that determines $P$, and in concrete situations this set should be chosen so that those conditions hold. In the setting of \cite{HW21b}, and in Section \ref{sec APS}, the conditions are satisfied if $J$ is the positive part of the spectrum of $D_N^+$.
See also Remark \ref{rem APS bdry cond} for comments on the role of APS-type boundary conditions.
\end{remark}

For $s>0$, let $e_{P}^{-s D_C^- D_C^+}$ be the heat operator for the operator \eqref{eq DC-DC+}.
We will see in Lemma \ref{lem smooth At} that the operator $e_{P}^{-s D_C^- D_C^+}D_C^-$ has a smooth kernel $ \lambda_s^P$.
For $a' \in (0,1)$, consider the \emph{contribution from infinity}
\beq{eq def At}
A_{g}(D_C, a') := -f_1(a') \int_0^{\infty}
\int_N \tr(g \lambda_s^P(g^{-1}n, a'; n, a'))\, dn\,
ds,
\eeq
defined whenever
 the integral in \eqref{eq def At} converges. If this is the case, then we say that ``$A_{g}(D_C, a')$ converges''.


For a finite-dimensional vector space $V$ with a given representation by $G$, we write $\tr(g|_V)$ for the trace of the action by $g$ on $V$.
The classical equivariant index of the $G$-equivariant Fredholm operator $D$ is
\[
\ind_G(D) := [\ker(D^+)] - [\ker(D^-)]
\]
 in the representation ring $R(G)$, and $\ind_G(D)(g)$ is the value of its character at $g$:
\[
\ind_G(D)(g) := \tr(g|_{\ker(D^+)}) - \tr(g|_{\ker(D^-)}).
\]

Let $\AS_g(D)$ be the Atiyah--Segal--Singer integrand associated to $D$, see for example Theorem 6.16 in \cite{BGV} or Theorem 3.9 in \cite{ASIII}. It is a differential form of mixed degree on the fixed-point set $M^g$ of $g$. The connected components of $M^g$ may have different dimensions, and the integral of $\AS_g(D)$ over $M^g$ is defined as the sum over these connected components of the integral of the component of the relevant degree.
\begin{theorem}[Index theorem for Dirac operators invertible at infinity]\label{thm index general}
For all $a' \in (0,\infty)$,
the contribution from infinity $A_{g}(D_C, a')$ converges, and
\beq{eq index general}
\ind_G(D)(g) = \int_{Z^g \cup (N^g \times (0,a'])}\AS_g(D)+A_{g}(D_C, a').
\eeq
\end{theorem}

The limit $a' \downarrow 0$ yields a version of Theorem \ref{thm index general} that is independent of $a'$.
\begin{corollary} \label{cor index general}
We have
\beq{eq index general 2}
\ind_G(D)(g) = \int_{Z^g}\AS_g(D)+ \lim_{a' \downarrow 0}A_{g}(D_C, a').
\eeq
\end{corollary}
\begin{proof}
The first term in the right hand side of \eqref{eq index general} approaches the first term on the right hand side of \eqref{eq index general 2} as $a' \downarrow 0$.
\end{proof}

\begin{remark}\label{rem zero at a}
In Corollary \ref{cor index general}, the limit $\lim_{a' \downarrow 0}A_{g}(D_C, a')$ exists, but it generally does \emph{not} equal the right hand side of \eqref{eq def At} with $a'$ replaced by $0$. In other words, the limit $a' \downarrow 0$ may not be taken inside the integral over $s$. Indeed, we will see in Proposition \ref{prop eF eP} that
\[
\int_N \tr(g \lambda_s^P(g^{-1}n, 0; n, 0)) = 0
\]
for all $n \in N$. For an explicit example, see Remark \ref{rem cyl zero at a}.
\end{remark}


\subsection{Special cases}\label{sec special cases}

An immediate consequence of Theorem \ref{thm index general} is the following equivariant version of \emph{Gromov and Lawson's relative index theorem} on page 304 of \cite{Gromov83}. This shows that, already  in the case where $G$ is trivial, Theorem \ref{thm index general} can be viewed as an absolute version of Gromov and Lawson's result, for manifolds with warped product structures at infinity.
\begin{corollary}
\label{cor GL}
For $j=1,2$, let $M_j$, $S_j$, $D_j$ and $Z_j$ be objects like $M$, $S$, $D$ and $Z$, respectively. Suppose that there is a $G$-equivariant isometry $M_1 \setminus  Z_1 \cong M_2 \setminus Z_2$ that preserves all structure.
%
%
%
Then
\beq{eq rel index}
\ind_G(D_1)(g) - \ind_G(D_2)(g) = \int_{Z_1^g} \AS_g(D_1) -
\int_{Z_2^g} \AS_g(D_2).
\eeq
\end{corollary}
\begin{proof}
In this setting, the number $A_{g}((D_1)_{C_1}, a')$ associated to $D_1$ as in \eqref{eq def At} equals the number $A_{g}((D_2)_{C_2}, a')$ associated to $D_2$ for all $a' \in (0,1)$, so the contributions from infinity in Corollary \ref{cor index general} cancel.
\end{proof}


In the special case where $C ={M \setminus Z}$ is equivariantly isometric to a product $N \times (0, \infty)$, Theorem \ref{thm index general} implies an equivariant version of the \emph{Atiyah--Patodi--Singer index theorem} \cite{APS1, Donnelly}. In that case, suppose that on $C$, the Dirac operator $D$ has a product form
\[
D|_C = \sigma \left( \frac{\partial}{\partial x}+D_N\right)
\]
for an odd vector bundle isomorphism $\sigma\colon S|_N \to S|_N$, an even Dirac operatoer $D_N$ on $S|_N$, and where $x$ is the coordinate in $(0, \infty)$. Suppose that $D_N$ is invertible; then $D$ is invertible outside $Z$ in the sense of \eqref{eq D inv infty}.
In this case, $\ind(D)$ is the equivariant APS-index of the restriction of $D$ to $Z$, which we denote by $\ind^{\APS}_G(D_Z)$.

Let $D_N^+$ be the restriction of $D_N$ to even-graded sections.
Its  \emph{$g$-delocalised $\eta$-invariant}  is
\[
\eta_g(D_N^+) := \frac{1}{\sqrt{\pi}} \int_0^{\infty} \Tr\bigl(g \circ e^{-s (D_N^+)^2}D_N^+ \bigr)\frac{1}{\sqrt{s}}\, ds,
\]
if this converges.
\begin{corollary}[Donnelly--APS index theorem]\label{cor APS}
The $g$-delocalised $\eta$-invariant of $D_N^+$ converges, and
\beq{eq cor APS}
\ind^{\APS}_G(D_Z)(g) = \int_{Z^g} \AS_g(D) -\frac{1}{2} \eta_g(D_N^+),
\eeq
\end{corollary}
This is the case of Theorem 1.2 in \cite{Donnelly} where $D_N^+$ is invertible.
If $D_N^+$ is not invertible, then this result can be obtained in a similar way,  through a spectral shift of $D_N^+$. See for example  \cite{Melrose}, Subsection 5.1 of \cite{HW21b} or Section 6 of \cite{HWW}.

Corollary \ref{cor APS} is deduced from Corollary \ref{cor index general} in Section \ref{sec APS}.

\section{Smooth kernels}\label{sec smooth kernels}

To prepare for the proof of Theorem \ref{thm index general} in Section \ref{sec proof index thm}, we obtain some results on certain smooth kernel operators. One of these is a relation between heat operators with APS and Dirichlet boundary conditions, Proposition \ref{prop eF eP}.

\subsection{Kernels and projections}\label{sec ker proj}

A key reason why the index of $D$ splits up into an interior component and a contribution from infinity, as in Theorem \ref{thm index general}, is a vanishing result for a certain difference of two parametices on $U$, Proposition \ref{prop lem 5.5}. In the proof of that proposition, we use the `principle of not feeling the boundary' to show that heat kernels defined with respect to different boundary conditions are equal in the short time limit. That principle applies to local, e.g.\ Dirichlet boundary conditions. To apply it in our setting, where we use APS boundary conditions, we analyse how the principle of not feeling the boundary interacts with the projection $P$ in \eqref{eq def P}. This is starts with Lemma \ref{lem ker proj}, which we prove in this subsection.

By identifying $S|_N^* \cong S|_N$ via the Hermitian metric on $S|_N$, we may consider the operator
$D_N\otimes D_N$ on $\Gamma^{\infty} (S|_N \boxtimes S|_N^*)$. For any $\kappa \in \Gamma^{\infty}(S|_N \boxtimes S|_N^*)$, we denote the corresponding bounded operator on $L^2(S|_N)$ by $\kappa$ as well. For such a $\kappa$,  the composition $ \kappa \circ P$, with $P$ as in \eqref{eq def P}, still has a smooth kernel. We denote this kernel by $\kappa \circ P$.

%
%
\begin{lemma}\label{lem ker proj}
There is a constant $C_1>0$, depending only on the dimension and volume of $N$,  such that the following holds.
Let $\kappa \in \Gamma^{\infty}(S|_N \boxtimes S|_N^*)$, and suppose that the corresponding operator is trace-class. Suppose that there  are an integer $s > \dim(N)/2$ and a $C_2>0$ such that for all $n,n' \in N$,
\beq{eq est D kappa}
\| (D_N \otimes D_N)^s\kappa(n,n')\| \leq C_2.
\eeq
Then
\beq{eq Tr P kappa}
\left|\int_N \tr ( g(\kappa \circ P) (g^{-1}n,n)) \, dn \right| \leq \Tr(|\kappa|) \leq C_1C_2.
\eeq
\end{lemma}
\begin{proof}
Let $\{e_j\}_{j=1}^{\infty}$ be a Hilbert basis of $L^2(S|_N)$ of eigensections of $D_N$, such that the absolute values of the corresponding eigenvalues $\lambda_j$ increase  (possibly non-strictly) in $j$.
For each $j$, let $e_j^* \in L^2(S|_N^*)$ be the section pointwise dual to $e_j$ with respect to the Hermitian metric.
Write
\[
\kappa = \sum_{j,k=1}^{\infty} \kappa^{j, k} e_j \otimes e_k^*,
\]
for $\kappa^{j,k} \in \C$. Then by \eqref{eq est D kappa},
\[
\sum_{j,k = 1}^{\infty} |\lambda_j|^{2s} |\lambda_k|^{2s} |\kappa^{j,k}|^{2}
	= \|(D_N \otimes D_N)^s \kappa\|_{L^2(S|_N \boxtimes S|_N^*)}^2
	\leq C_2^2 \vol(N)^2.
\]
In particular, we have for all $j,k$, such that $\lambda_j$ and $\lambda_k$ are nonzero,
\beq{eq growth kappa}
|\kappa^{j,k}| \leq C_2 \vol(N)  |\lambda_j|^{-s} |\lambda_k|^{-s}.
\eeq

By Weyl's law, there is a constant $A>0$, only depending on the volume and dimension of $N$,  such that for all $l$ for which $\lambda_{l} \not=0$,
\[
|\lambda_l| \geq A l^{1/\dim(N)}.
%
\]
Combining this with \eqref{eq growth kappa}, we find that for all but finitely many $j$ and $k$,
\beq{eq est kappa 2}
|\kappa^{j,k}| \leq C_2 \vol(N) A^{-2s} (jk)^{-s/\dim(N)}. 
\eeq

Because $\kappa$ is trace-class, so is
 $ \kappa \circ P$. Hence the left hand side of \eqref{eq Tr P kappa} equals
\[
|\Tr(g\circ \kappa \circ P)| \leq \Tr(|g\circ \kappa \circ P|) \leq  \|g\| \|P\|\Tr(|\kappa|) = \Tr(|\kappa|).
\]
Because $2s/\dim(N)>1$, the inequality \eqref{eq est kappa 2} implies that
\[
\Tr(|\kappa|) =
\sum_{j=1}^{\infty} |\kappa^{j,j}| \leq C_1C_2,
\]
for a constant $C_1$ only depending on the volume and dimension of $N$.
\end{proof}

\begin{remark}
The $g$-trace of Definition \ref{def g trace} can be generalised to proper actions by noncompact groups, see for example Subsection 2.1 of \cite{HWW}. The main reason why we assume $G$ to be compact in the current paper is the fact that the $g$-trace of $\kappa \circ P$ on the left hand side of \eqref{eq Tr P kappa} then equals $\Tr(g \circ \kappa \circ P)$, as used in the proof of Lemma \ref{lem ker proj}.
\end{remark}

\subsection{Heat kernels and boundary conditions}\label{sec heat bdry}

Let $e_{1-P}^{-tD_C^+ D_C^-}$ be the heat operator for \eqref{eq DC+DC-}. Recall that $e_{P}^{-t D_C^- D_C^+}$ is the heat operator for  \eqref{eq DC-DC+}.
\begin{lemma}\label{lem heat P 1-P}
We have
\[
e_{P}^{-tD_C^- D_C^+} D_C^- = D_C^- e_{1-P}^{-tD_C^+ D_C^-}.
\]
\end{lemma}
\begin{proof}
Let $s \in \Gamma^{\infty}_c(S_C^+)$. For $t>0$, we write
\[
\begin{split}
\sigma_t^P &:=  e_{P}^{-tD_C^- D_C^+} D_C^- s;\\
\sigma_t^{1-P} &:= D_C^- e_{1-P}^{-tD_C^+ D_C^-} s.
\end{split}
\]
The claim is that $\sigma^P_t = \sigma^{1-P}_t$ for all such $s$ and all $t>0$.

The family of sections $\sigma_t^P$ is the unique solution of the equations
\begin{align}
\frac{\partial \sigma_t}{\partial t} &= D_C^- D_C^+ \sigma_t; \label{eq sigma t I}\\
\sigma_0 &= D_C^- s; \label{eq sigma t II}\\
P(  \sigma_t |_N) &= 0; \label{eq sigma t III}\\
(1-P)( D_C^+  \sigma_t |_N) &= 0. \label{eq sigma t IV}
\end{align}
And the family $\sigma^{1-P}_t$ is the unique solution of \eqref{eq sigma t I}, \eqref{eq sigma t II}, \eqref{eq sigma t III} and
\beq{eq sigma t V}
(1-P)( (D_C^-)^{-1}  \sigma_t |_N) = 0.
\eeq
Here we used the assumption that the operator \eqref{eq DC+} is invertible.
We claim that \eqref{eq sigma t I}, \eqref{eq sigma t II} and \eqref{eq sigma t IV} are equivalent to \eqref{eq sigma t I}, \eqref{eq sigma t II} and \eqref{eq sigma t V}, which implies that $\sigma^P_t = \sigma^{1-P}_t$ for all $t$.

Suppose first that \eqref{eq sigma t I}, \eqref{eq sigma t II} and \eqref{eq sigma t V} hold. By \eqref{eq sigma t I}  and \eqref{eq sigma t V},
\[
\begin{split}
(1-P)( D_C^+  \sigma_t |_N)& = (1-P)( (D_C^-)^{-1} D_C^- D_C^+  \sigma_t |_N) \\
&= \frac{\partial}{\partial t}(1-P)( (D_C^-)^{-1}   \sigma_t |_N) \\
&= 0.
\end{split}
\]
so \eqref{eq sigma t IV} holds.

Next, suppose  that \eqref{eq sigma t I}, \eqref{eq sigma t II} and \eqref{eq sigma t IV} hold. Then \eqref{eq sigma t I} and \eqref{eq sigma t IV} imply that
\[
\frac{\partial}{\partial t}(1-P)( (D_C^-)^{-1}   \sigma_t |_N)  = (1-P)( D_C^+  \sigma_t |_N) = 0.
\]
And by \eqref{eq sigma t II},
\[
(1-P)( (D_C^-)^{-1}   \sigma_0 |_N) = P(s|_N) = 0.
\]
Here we used that $s$ is supported in the interior of $C$.
We conclude that $P( (D_C^-)^{-1}   \sigma_t |_N) =0$ for all $t$, which is to say that  \eqref{eq sigma t V} holds.
\end{proof}


\subsection{Dirichlet heat operators}\label{sec heat Dirichlet}

The operator $D_C^2$ on $\Gamma^{\infty}_c(S_C)$ is symmetric and nonnegative. Hence its Friedrichs extension is self-adjoint. Its domain is defined by Dirichlet boundary conditions:
\[
W^2_D(S_C; F) := \{s \in W^2_D(S_C); s|_N = 0\}.
\]
We denote the restrictions of the heat operator for this Friedrichs extension to even and odd-graded sections by $e^{-t D_C^- D_C^+}_F$ and $e^{-t D_C^+ D_C^-}_F$, respectively.

Let $p:= \dim(M)$.  
Consider the bounded operator
\begin{multline} \label{eq P zeta}
P \colon L^2(S_C) \cong L^2(S|_N) \otimes L^2( (0,\infty), e^{p\varphi}dx )\\
 \xrightarrow{P \otimes 1} \bigoplus_{\lambda \in J}  L^2(S|_N)_{\lambda} \otimes L^2( (0,\infty), e^{p\varphi}dx ) \cong L^2(S_C).
\end{multline}
The first and third maps are $G$-equivariant, unitary isomorphisms.
Here we use the fact that by the form \eqref{eq metric U} of the Riemannian metric on $C$, the Riemannian density on $C$ is $d\vol_{g_N} \otimes e^{p\varphi} dx$.

\begin{lemma}\label{lem Dirichlet}
For all $t>0$,
\begin{align}
e_P^{-t D_C^-D_C^+} P &= e_F^{-t D_C^-D_C^+} P; \label{eq Dirichlet P}\\
e_{1-P}^{-t D_C^+D_C^-} (1-P) &= e_F^{-t D_C^+D_C^-} (1-P).\label{eq Dirichlet 1-P}
\end{align}
\end{lemma}
\begin{proof}
We have
\beq{eq P 1-P F}
\begin{split}
\im(P) \cap W^2_D(S_C^+; P) &= \im(P\zeta) \cap W^2_D(S_C^+; F);\\
\im(1-P) \cap W^2_D(S_C^-; 1-P) &= \im( (1-P)\zeta) \cap W^2_D(S_C^-; F).
\end{split}
\eeq
So the operator \eqref{eq DC-DC+} equals the Friedrichs extension of $D_C^-D_C^+$ on the image of $P$, and the operator \eqref{eq DC+DC-} equals the Friedrichs extension of $D_C^+D_C^-$ on the image of $1-P$. 

For every eigenvalue $\lambda$ of $D_N^{\pm}$, the operator $D_C^{\pm}$ restricted to 
 $L^2(S^{\pm}|_N)_{\lambda} \otimes L^2( (0,\infty), e^{p\varphi}dx )$ equals $\sigma \otimes  (D_C^{\pm})_{\lambda}$, with
 \[
 (D_C^{\pm})_{\lambda} := f_1 \frac{\partial}{\partial x} + \lambda f_2 + f_3.
 \]
 Let $e_{F}^{-t (D_C^{\pm})_{\lambda}(D_C^{\mp})_{\lambda} }$ be the corresponding heat operator on $ L^2( (0,\infty), e^{p\varphi}dx )$, with Dirichlet boundary conditions at $0 \in [0, \infty)$. Then because of \eqref{eq P 1-P F} and the fact that $\sigma^2 = 1$, both sides of \eqref{eq Dirichlet P} equal
 \[
 \bigoplus_{\lambda \in J} \Id_{L^2(S^{+}|_N)_{\lambda}} \otimes e_{F}^{-t (D_C^{-})_{\lambda}(D_C^{+})_{\lambda} },
 \]
while  both sides of \eqref{eq Dirichlet 1-P} equal
 \[
 \bigoplus_{\lambda \not\in J}   \Id_{L^2(S^{-}|_N)_{\lambda}} \otimes e_{F}^{-t (D_C^{+})_{\lambda}(D_C^{-})_{\lambda} }.
  \]

%
%
\end{proof}

Lemmas \ref{lem heat P 1-P} and \ref{lem Dirichlet} imply the following equality, which is used in the proofs of Propositions \ref{prop lem 5.5} and \ref{prop conv cusp contr}. This equality can also be used to obtain a more computable expression for \ref{eq def At}, which does not include the APS-type heat operator $e_P^{-t D_C^- D_C^+}$. This is used in \cite{HW21b}.
\begin{proposition}\label{prop eF eP}
We have
\[
e_P^{-t D_C^- D_C^+}D_C^- = e_F^{-t D_C^- D_C^+}D_C^-P + D_C^-e_F^{-t D_C^+ D_C^-}(1-P).
\]
\end{proposition}

\subsection{Operators on a double}

Let $\tilde M$ be a compact Riemannian manifold, on which $G$ acts isometrically. Let  $\tilde S \to \tilde M$ be a $G$-equivariant Clifford module, and  $\tilde D$ a $G$-equivariant Dirac operator acting on sections of $\tilde S$.
Let $U\subset C$ be the $G$-invariant open subset that is the image of $N \times (0,2)$ under the isometry \eqref{eq isom U}. 
Let $V \subset U \cong N \times (0,\infty)$ be the image of $N \times (0,1)$.
Suppose that
$Z \cup V$
 embeds equivariantly and  isometrically into $\tilde M$, in such a way that $S|_{Z \cup V} = \tilde S|_{Z \cup V}$ and $D|_{Z \cup V} = \tilde D|_{Z \cup V}$. Such objects can be constructed as follows. Consider the manifold $Z \cup U$ with a Riemannian metric that equals the metric on $M$ when restricted to $Z \cup V$, and that equals the product metric $g_N + dx^2$ on $N \times (3/2, 2) \subset U$. Then take the double of this manifold along $N \times (3/2, 2) \subset U$, and extend $S$ and $D$ to it.

The operator $\tilde D$ is self-adjoint and odd-graded. So we have the heat operators  $e^{-t\tilde D^{\mp} \tilde D^{\pm}}$. Let $\tilde \kappa_t^{\pm}$ be the Schwartz kernel of $e^{-t\tilde D^{\mp} \tilde D^{\pm}}$, and
let $\kappa_t^{F, \pm}$ be the Schwartz kernel of $e_F^{-tD_C^{\mp}D_C^{\pm}}$. Consider the kernel
\beq{eq def Kt}
K_t^{\pm} := \kappa_t^{F, \pm}|_{U \times U} - \tilde \kappa_t^{\pm}|_{U \times U}.
\eeq

The principle of not feeling the boundary states that away from a boundary where Dirichlet boundary conditions are imposed,    Dirichlet heat kernels approximate heat kernels defined without boundary conditions in the short time limit. We will use the following version of this principle.
\begin{theorem}[Principle of not feeling the boundary] \label{thm not feel bdry}
For all $j, k,l \in \Z_{\geq 0}$,
\beq{eq not feel bdry}
\begin{split}
\| (D_N^k \otimes D_N^l)   (K_t^+ \circ D^-)\| &= \cO(t^j) \quad \text{and}\\
\| (D_N^k \otimes D_N^l)   (D^- \circ K_t^+)\| & = \cO(t^j),
\end{split}
\eeq
as $t\downarrow 0$, uniformly in compact subsets of $U \times U$.
Here $K_t^+ \circ D^-$ is the kernel of the composition of $K_t^+$ with $D^-$, and $D^- \circ K_t^+$ is defined similarly. The operators $D_N^k \otimes D_N^l$ are applied in the $N$-direction in $U = N \times (0,2)$.
\end{theorem}
\begin{proof}
We first compare  the heat kernel $\kappa_t^F$ of  $e_F^{-tD_C^{2}}$ with the heat kernel $\kappa_t$ of $e^{-tD^2}$ on $M$.
By the principle of not feeling the boundary for the operator $D$, there are constants $a,b>0$ such that
 for all 
and all $t \in (0,1]$,
\[
\|\kappa_t^F  - \kappa_t\| \leq a e^{-b/t},
\]
uniformly in compact subsets of $U$.
See Theorem 1.5 in \cite{Hsu93}, Theorem 2.1 in \cite{Dodziuk97} or Theorem 2.26 in \cite{Lueck99}. 
Hence, in particular, the two heat kernels have the same asymptotic expansion in $t$ on $U$.
We can also show this directly, by noting that the terms in these asymptotic expansions are determined by local geometry (see for example Theorem 7.15(i) in \cite{Roe98}), and are therefore equal for $\kappa_t^F$ and $\kappa_t$ on $U$, analogously to the proof of  Lemma 5.4 in \cite{HWW}.

 Because asymptotic expansions commute with derivatives (see for example Theorem 7.15(ii) in \cite{Roe98}), it follows that for all $j,k,l$,
\[
\begin{split}
\| (D_N^k \otimes D_N^l)   ((\kappa_t^F - \kappa_t)\circ D^-)\| & =\cO(t^j) \quad \text{and}\\
\| (D_N^k \otimes D_N^l)   (D^- \circ (\kappa_t^F - \kappa_t))\| &= \cO(t^j),
\end{split}
\]
uniformly in compact subsets of $U$.

A similar estimate holds with $\kappa_t^F$ replaced by $\tilde \kappa_t$, analogously to Lemma 5.4 in \cite{HWW}. Hence the claim follows by the triangle inequality.
\end{proof}

We will only need the weaker version of Theorem \ref{thm not feel bdry} that the left hand sides of \eqref{eq not feel bdry} are bounded as $t \downarrow 0$.

%

\subsection{Parametrices}\label{sec parametrices}

Recall that
$V \cong N \times (0,1) \subset U$.
Let $\varphi_1, \varphi_2, \psi \in C^{\infty}(M)^G$ be $G$-invariant functions with values in $[0,1]$,  such that
\begin{itemize}
\item $\psi|_U$, $\varphi_1|_U$ and $\varphi_2|_U$ are constant in the $N$-component of $U = N \times (0,2)$;
\item $\varphi_1 \psi = \psi$ and $\varphi_2(1- \psi) = 1- \psi$ ;
\item $\psi|_{Z} = 1$ and $\psi|_{M \setminus V} = 0$;
\item $\varphi_1|_{M \setminus V} = 0$ and $\varphi_2|_{Z} = 0$.
\end{itemize}
The second and third conditions imply that $\varphi_1|_{Z} = 1$ and $\varphi_2|_{M \setminus V} = 1$.

Consider the parametrix
\beq{eq def Q tilde}
\tilde Q := \frac{1-e^{-t\tilde D^- \tilde D^+}
}{\tilde D^- \tilde D^+}  \tilde D^-
\eeq
of $\tilde D^+$. Here the first factor is defined as the functional calculus construction $f(\tilde D^- \tilde D^+)$ for the bounded continuous function $f(y) = \frac{1-e^{-ty}}{y}$ on $[0,\infty)$. So it is not assumed that $\tilde D^- \tilde D^+$ is invertible.

Let $(D_C^+)^{-1}$ be the inverse of the operator \eqref{eq DC+}. Consider the operators
\[
\begin{split}
R&:= \varphi_1 \tilde Q \psi + \varphi_2 (D_C^+)^{-1} (1-\psi);\\
R'&:= \psi \tilde Q \varphi_1 + (1-\psi) (D_C^+)^{-1} \varphi_2.
\end{split}
\]
Consider the remainder terms
\[
\begin{split}
\tilde S_+&:= 1-\tilde Q \tilde D^+ = e^{-t \tilde D^- \tilde D^+};\\
\tilde S_-&:= 1- \tilde D^+\tilde Q = e^{-t \tilde D^+ \tilde D^-},
\end{split}
\]
and
\beq{eq Spm}
\begin{split}
S_+&:= 1-R  D^+ ;\\
S'_+&:= 1-R'  D^+ ;\\
S_-&:= 1-  D^+ R.
\end{split}
\eeq
We have the following generalisation of Lemma 5.1 in \cite{HWW}.
\begin{lemma}\label{lem Sj}
With $f_1$ as in \eqref{eq D on U},
\[
\begin{split}
S_+ &= \varphi_1 \tilde S_{+} \psi + \varphi_1 \tilde Q \sigma \psi' f_1 - \varphi_2 (D_C^+)^{-1} \sigma \psi' f_1;\\
S_+' &= \psi \tilde S_{+} \varphi_1 + \psi \tilde Q \sigma \varphi_1' f_1+ (1-\psi) (D_C^+)^{-1} \sigma \varphi_2' f_1;\\
S_- &= \varphi_1 \tilde S_{-} \psi - f_1 \varphi_1' \sigma \tilde Q \psi - f_1 \varphi_2' \sigma (D_C^+)^{-1} (1-\psi).
\end{split}
\]
\end{lemma}
\begin{proof}
This is a direct computation, based on the fact that $\tilde D$ and $D_C$ both equal \eqref{eq D on U} in the region where the functions $\varphi_j$ and $\psi_j$ are not constant. The difference with Lemma 5.1 in \cite{HWW} and its proof is that now
\[
[D,\varphi] = \varphi' f_1 \sigma
\]
for a smooth function $\varphi \in C^{\infty}(0,\infty) \hookrightarrow C^{\infty}(U)$, whereas in Lemma 5.1 in \cite{HWW}, the function $f_1$ equals one.
\end{proof}

The operators $S_{\pm}$ and $S_{+}'$ have smooth kernels by the same argument as the proof of Lemma 5.2 in \cite{HWW}.

The inverse $(D_C^+)^{-1}$ is a parametrix of $D_C^+$, but we will also use the parametrix
\[
Q_C := \frac{1-e_P^{-t D^-_C  D^+_C}
}{D_C^- D_C^+}  D_C^-
\]
of $D_C^+$, defined via functional calculus like \eqref{eq def Q tilde}.

\section{The contribution from infinity} \label{sec proof index thm}

After the preparations in Section \ref{sec smooth kernels}, we prove Theorem \ref{thm index general} in this section, by showing how the contribution from infinity $A_{g}(D_C, a')$ appears. We first obtain a version of Theorem \ref{thm index general} in which the contributions to the index from near $Z$ and from infinity are divided by a smooth cutoff function, this is Theorem \ref{thm index general psi}. Then we deduce Theorem \ref{thm index general} by letting this smooth function approach a step function in a suitable way.

The proof of  Theorem \ref{thm index general psi}  is a generalisation of the arguments in Section 5 of \cite{HWW}.
We first investigate convergence of a version of $A_{g}(D_C, a')$ that involves a smooth cutoff function, and then show how the results from Section \ref{sec smooth kernels} can be combined with arguments from \cite{HWW} to obtain  Theorem \ref{thm index general psi}.


\subsection{The $g$-trace and the $g$-index}

As in \cite{HWW, HW2}, we will use an extension of the operator trace to operators that are not necessarily trace-class.


\begin{definition}\label{def g trace}
Let $T$ be an operator on $\Gamma^{\infty}(S)$ with a smooth kernel $\kappa$.
Then $T$ is \emph{$g$-trace class} if the integral
\[
\int_M \tr \bigl( g\kappa(g^{-1}m,m) \bigr)\, dm
\]
converges absolutely. In that case, the value of this integral is the \emph{$g$-trace} of $T$, denoted by $\Tr_g(T)$.
\end{definition}
If $T$ is trace class and has a smooth kernel, then it is $g$-trace class, and $\Tr_g(T) = \Tr(g\circ T)$.

The $g$-trace has a trace property, see Lemma 3.2 in \cite{HWW}.
\begin{lemma}\label{lem Trg trace}
Let $S,T$ be $G$-equivariant linear operators on  $\Gamma^{\infty}(S)$. Suppose that
\begin{itemize}
\item $S$ has a distributional  kernel;
\item $T$ has a smooth kernel;
\item $ST$ and $TS$ are $g$-trace-class.
\end{itemize}
Then $\Tr_g(ST) = \Tr_g(TS)$.
\end{lemma}

\begin{definition}
Let $F$ be a differential operator on $\Gamma^{\infty}(S)$. Suppose that $F$ is odd with respect to the grading on $S$, and let $F^+\colon \Gamma^{\infty}(S^+) \to  \Gamma^{\infty}(S^-)$ be its restriction to even-graded sections. Then $F$ is \emph{$g$-Fredholm} if
$F^+$ has a parametrix $Q$ such that $1_{S_+} - QF^+$ and $1_{S_-} - F^+ Q$ are $g$-trace class. In that case, the \emph{$g$-index} of $F$ is
\[
\ind_g(F) := \Tr_g(1_{S_+} - QF^+) - \Tr_g(1_{S_-} - F^+ Q).
\]
\end{definition}

The  properties of the $g$-index in the following lemma are special cases of  Lemmas 2.5 and 2.9 in \cite{HWW}.
\begin{lemma}\label{lem inde ind}
The $g$-index of a $g$-Fredholm operator is independent of the choice of the parametrix $Q$. If $F$ has finite-dimensional kernel and is $g$-Fredholm, then
\[
\ind_g(F) = \ind_G(F)(g).
\]
\end{lemma}

\subsection{Convergence of a contribution from infinity}

Let the functions $\varphi_j$ and $\psi$ be as at the start of Subsection \ref{sec parametrices}, and $f_1$ as in \eqref{eq D on U}.
The following result is a variation on Lemma 5.5 in \cite{HWW}. Because of the boundary conditions used to define the domains of the operators \eqref{eq DC-DC+} and \eqref{eq DC+DC-}, we now use the arguments in Subsections \ref{sec ker proj}--\ref{sec heat Dirichlet} to prove this.
\begin{proposition} \label{prop lem 5.5}
The operator $(\varphi_1 \tilde Q - \varphi_2 Q_C) \sigma \psi' f_1$ is $g$-trace class for all $t>0$, and
\[
\lim_{t\downarrow 0} \Tr_g \bigl( (\varphi_1 \tilde Q - \varphi_2 Q_C) \sigma \psi' f_1 \bigr) = 0.
\]
\end{proposition}
\begin{proof}
Let $\varphi \in C^{\infty}_c(U)$ be such that $\varphi \psi' = \psi'$, and
for $j=1,2$, 
 $\varphi$ is constant zero outside the support of $1-\varphi_j$.
As in the proof of Lemma 5.5 in \cite{HWW},
\[
\Tr_g \bigl( (\varphi_1 \tilde Q - \varphi_2 Q_C) \sigma \psi' f_1\bigr)  =
-\int_0^t \Tr_g\left( \varphi \left(e^{-s \tilde D^- \tilde D^+} \tilde D^- - e_P^{-s D_C^- D_C^+} D_C^- \right) \sigma \psi' f_1 \right) \, ds.
\]
By 
Proposition \ref{prop eF eP},
 and the fact that $\tilde D|_U = D_C|_U = D|_U$, this equals
\begin{multline} \label{eq kernels P 1-P}
-\int_0^t \Tr_g\left( \varphi \left(e^{-s \tilde D^- \tilde D^+}- e_F^{-s D_C^- D_C^+} \right) P D^-\sigma \psi' f_1\, \right)ds\\
-\int_0^t  \Tr_g\left(  \varphi D^- \left(e^{-s \tilde D^+ \tilde D^-}- e_F^{-s D_C^+ D_C^-} \right) (1-P)\sigma \psi' f_1 \right) \, ds.
\end{multline}
To be precise, in the first term we use the equality
\[
D^- \sigma \psi' = \sigma \psi' D^- + \sigma c(d\psi')
\]
and apply Proposition \ref{prop eF eP} twice, with $\zeta = \sigma \psi'$ and with $\zeta = \sigma c(d\psi')$.

The absolute value 
 of the first term in \eqref{eq kernels P 1-P} is at most equal to
\beq{eq int 1st term}
\int_0^t \int_0^2 |\varphi(x)\psi'(x) f_1'(x)| \left| \int_N \tr\left( g K^+_sD^-\sigma P (g^{-1}n,x; n,x)\right)\, dn \right|
\, dx
\, ds
\eeq
Here $K^+_s$ is as in \eqref{eq def Kt}. By Theorem \ref{thm not feel bdry}, the kernel $K^+_sD^-\sigma$ satisfies the conditions of Lemma \ref{lem ker proj}, where the constant $C_2$ is independent of $s$. That lemma therefore implies that the integrand in \eqref{eq int 1st term} is bounded in $s$. Hence the integral converges, and goes to zero as $t \downarrow 0$.

The argument for the second term in \eqref{eq kernels P 1-P} is similar, where $P$ is replaced by $1-P$.
\end{proof}

For $t>0$, we write
\[
A_{g}^t(D_C, \psi') := \int_t^{\infty} \Tr_g\bigl(\varphi_2 e_P^{-s D_C^- D_C^+} D_C^- \psi' f_1\bigr)\, ds,
\]
when this converges.
\begin{proposition}\label{prop conv cusp contr}
If the operators \eqref{eq DC+} and \eqref{eq DC-} are invertible, then $A_{g}^t(D_C, \psi')$
converges
 for all $t>0$.
\end{proposition}

\begin{lemma}\label{lem smooth At}
 The operator $e_P^{-sD_C^-D_C^+}D_C^-$ has a smooth kernel.
\end{lemma}
\begin{proof}
%
%
For any $k \in \N$, we can form the bounded operator
\[
(D_C^-D_C^+)^k
e_P^{-s D_C^- D_C^+}
\]
on $L^2(S_C^+)$,
by applying functional calculus to \eqref{eq DC-DC+}.
So for all such $k$, and all $\xi \in L^2(S_C^+)$, the section $(D_C^-D_C^+)^k(e^{-s D_C^- D_C^+} \xi)$ is in $L^2(S_C^+)$. Hence $e^{-s D_C^- D_C^+} \xi$ is smooth, because $D_C^-D_C^+$ is elliptic. So $e_P^{-s D_C^- D_C^+}$ has a smooth kernel, and therefore so does $e_P^{-sD_C^- D_C^+}D_C^-$.
%
\end{proof}


We will use an analogue of  Theorem 10.24 in \cite{Grigoryan09}.
\begin{lemma}\label{lem est kappa F}
Suppose that $D_C^2 \geq b^2$ for $b>0$.
Then for all $k_1, k_2, k_3 \in \Z_{\geq 0}$,  there is a $B>0$ such that on $U$, for all $s \geq 1$,
\[
\begin{split}
\| ( D_N^{k_1} \otimes D_N^{k_2})D_C \frac{\partial^{k_3}}{\partial s^{k_3}}\kappa_s^{F}\| & \leq B e^{-b^2s} \quad \text{and}\\
\| ( D_N^{k_1} \otimes D_N^{k_2})\frac{\partial^{k_3}}{\partial s^{k_3}} \kappa_s^{F}D_C\| & \leq B e^{-b^2s}.
\end{split}
\]
\end{lemma}
\begin{proof}
For a relatively compact open subset $\Omega \subset C$, let $D_{\Omega}^2$ be the Friedrichs extension of $D^2|_{\Gamma_c^{\infty}(S|_{\Omega})}$. Then $D_{\Omega}^2$ has discrete spectrum satisfying the Weyl law. Let $\lambda_1(\Omega)$ be the smallest eigenvalue. Let $\kappa_s^{F, \Omega}$ be the  heat kernel associated to $D_{\Omega}^2$.  Then for all $k_1, k_2, k_3 \in \Z_{\geq 0}$ there is a constant $C_{\Omega}$ such that for all $n,n' \in N$ and $x,x' \in (0,\infty)$,
\beq{eq est Omega}
\| ( D_N^{k_1} \otimes D_N^{k_2})D \frac{\partial^{k_3}}{\partial s^{k_3}}\kappa_s^{F, \Omega}(n, x; n', x')\| \leq C_{\Omega} e^{-s \lambda_1(\Omega)},
\eeq
and a similar estimate holds with $D\kappa_s^{F, \Omega}$ replaced by $\kappa_s^{F, \Omega} D$.
See e.g.\ 
the second-last displayed equation of the proof of Theorem 10.24 in \cite{Grigoryan09}
 for the scalar case, and without differential operators applied to the heat kernel. The argument is relatively elementary and also applies in our setting. Indeed, \eqref{eq est Omega} follows directly from a decomposition of $\kappa_s^{F, \Omega}$ into eigensections of $D_{\Omega}^2$, as in the third-last displayed equation of the proof of Theorem 10.24 in \cite{Grigoryan09}. Each derivative with respect to $s$ can be replaced by  $D_C^2$ because $\kappa_s^F$ satisfies the heat equation.

 One has $\lambda_1(\Omega) \leq \lambda_1(\Omega')$ when $\Omega' \subset \Omega$, so then $e^{-s \lambda_1(\Omega')} \leq e^{-s \lambda_1(\Omega)}$. This implies that  for all such $\Omega$,
\[
e^{-s \lambda_1(\Omega)} \leq e^{-b^2s},
\]
and that the lemma is true.
\end{proof}

\begin{proof}[Proof of Proposition \ref{prop conv cusp contr}.]
Let $t \in (0,1]$.
Let $\kappa_t^{P, \pm}$ be the Schwartz kernel of $e_P^{-t D_C^{\mp}D_C^{\pm}}$.
As in the proof of Proposition \ref{prop lem 5.5}, we estimate
\begin{multline}\label{eq est Tre At}
 \int_t^{\infty}  \int_{\overline{M \setminus Z}} \left| \tr
\varphi_2 (g \kappa_s^{P, +}D_C^-) (g^{-1}m,m)\psi' f_1
 \right|\, ds \, dm \\
 \leq  \int_t^{\infty} \int_0^2  |\psi'(x) f_1(x)| \left| \int_N
  \tr (g \kappa_s^{P, +}D_C^-) (g^{-1}n, x;  n, x)\, dn \right| \,dx \, ds\\
  =
  \int_t^{\infty} \int_0^2  |\psi'(x) f_1(x)| \left| \int_N
  \tr \left(g\kappa_s^{F, +}D_C^- P + gD_C^-\kappa_s^{F, -}(1- P)\right) (g^{-1}n, x;  n, x)\, dn \right| \,dx \, ds.
%
\end{multline}
Here we used  Proposition \ref{prop eF eP}. 

By Lemma \ref{lem est kappa F}, we may apply Lemma \ref{lem ker proj} to find that there is a constant $B>0$ such that for all $x \in (0,\infty)$ and all $s\geq 1$,
\[
\left|
 \int_N
 \tr\left(g\kappa_s^{F, +}D_C^- P + D_C^-\kappa_s^{F, -}(1- P)\right) (g^{-1}n, x;  n, x) \, dn\right| \leq Be^{-b^2 s}.
\]
So the integral in the right hand side of \eqref{eq est Tre At} 
 converges.
\end{proof}

We will deduce convergence of $\lim_{t\downarrow 0} A_{g}^t(D_C, \psi')$ from the index formula in the proof of Theorem \ref{thm index general psi} below. But as an aside, we note that this convergence also follows from a generalisation by Zhang of a result by Bismut and Freed, in the case of twisted $\Spinc$-Dirac operators.
\begin{proposition}
If  $D$ is a twisted $\Spinc$-Dirac operator, then $\lim_{t\downarrow 0} A_{g}^t(D_C, \psi')$ converges.
\end{proposition}
\begin{proof}
By Proposition \ref{prop conv cusp contr} and \eqref{eq est Tre At}, it is enough to show
 that the limit
\[
\lim_{t\downarrow 0}
  \int_t^{1} \int_0^2  |\psi'(x) f_1(x)| \left| \int_N
  \tr \left(g\kappa_s^{F, +}D_C^- P + D_C^-\kappa_s^{F, -}(1- P)\right) (g^{-1}n, x;  n, x)\, dn \right| \,dx \, ds
\]
converges. By Theorem \ref{thm not feel bdry} and Lemma \ref{lem ker proj}, we may replace $\kappa_s^{F, \pm}$ by $\tilde \kappa_s^{\pm}$ here.
And
\[
\tilde \kappa_s^+ D_C^- |_U = \tilde \kappa_s^+ \tilde D^-|_U = \tilde D^-  \tilde \kappa_s^- |_U.
\]
Arguing similarly for the second term, we find
%
%
\[
\left(
\tilde \kappa_s^{+}D_C^- P + D_C^-\tilde \kappa_s^{-}(1- P)\right) |_U =
\tilde D^- \tilde  \kappa_s^{-}|_U.
\]
We now use the assumption that $D$, and hence $\tilde D$ is a twisted $\Spinc$-Dirac operator. Then (2.2) in \cite{Zhang90}, generalising the second part of Theorem 2.4 in \cite{Bismut86}, states that
\[
\tr  \bigl( g\tilde D^- \tilde \kappa_s^{-} (g^{-1}n, x;  n, x)\bigr) = \cO(s^{1/2})
\]
as $s \downarrow 0$, uniformly in $(n,x) \in U$.
\end{proof}

\subsection{Appearance of a contribution from infinity}

\begin{lemma}\label{lem Trace exp}
Let $\chi \in C^{\infty}_c(0,\infty)$. Suppose that $D_C^2 \geq b^2>0$.
There is a   $B>0$ such that for all $s>0$,
\beq{eq Trace exp}
\Tr(|\chi e_P^{-s D_C^- D_C^+} D_C^- \psi' f_1|) \leq Be^{-b^2s}.
\eeq
\end{lemma}
\begin{proof}
There is a bounded linear isomorphism $L^2(S|_U) \cong L^2(S|_N) \otimes L^2(0,2)$ with bounded inverse. See Lemma 3.5 in \cite{HW21b} 
for a stronger statement (a unitary isomorphism), but in the current setting boundedness of $e^{\varphi}$ above and below on $(0,2)$ is enough. We work with $L^2(S|_N) \otimes L^2(0,2)$ from now on. Let $\{e_j^N\}_{j=1}^{\infty}$ be a Hilbert basis of $L^2(S|_N)$ of eigensections of $D_N$, and let $\{e_j^{(0,2)}\}_{j=1}^{\infty}$ be a Hilbert basis of $L^2(0,2)$. Then the sections $e_{j, k} := e_j^N \otimes e_k^{(0,2)}$ form a Hilbert basis of $L^2(S|_N) \otimes L^2(0,2)$.

%
We use Proposition \ref{prop eF eP} to write
\[
e_P^{-s D_C^- D_C^+} D_C^- \psi' = \bigl( e_F^{-s D_C^- D_C^+} D_C^- P + D_C^- e_F^{-s D_C^+ D_C^-} (1-P)\bigr) \psi'
\]
So the left hand side of \eqref{eq Trace exp} equals
\begin{multline}\label{eq Trace exp 2}
\sum_{j,k=1}^{\infty} \left| \bigl( e_j^N \otimes e_k^{(0,2)},
\chi e_P^{-s D_C^- D_C^+} D_C^- \psi' f_1 e_j^N \otimes e_j^{(0,2)}
\bigr)_{L^2(S|_N) \otimes L^2(0,2)} \right| \\
\leq
\max_{x,x' \in (0,2)} \Tr\left(\left| \left(
\kappa_s^{F, +}D_C^-P + D_C^- \kappa_s^{F, -}(1-P)
\right)(\relbar, x; \relbar, x') \right| \right)\\
\sum_{k=1}^{\infty} \left| 
(e_k^{(0,2)},  \chi)_{L^2(0,2)} (e_k^{(0,2)}, \psi' f_1)_{L^2(0,2)}
\right|.
\end{multline}
%
%
%
%
We apply Lemma \ref{lem est kappa F}, and use the resulting bound
 to apply Lemma  \ref{lem ker proj} to the first factor on the right. The second inequality in \eqref{eq Tr P kappa}  then yields a constant $B_0>0$ such that
 \[
 \max_{x,x' \in (0,2)} \Tr\left(\left| \left(
\kappa_s^{F, +}D_C^-P + D_C^- \kappa_s^{F, -}(1-P)
\right)(\relbar, x; \relbar, x') \right| \right) \leq  B_0e^{-b^2 s}.
 \]
 Furthermore,
 \[
 \sum_{k=1}^{\infty} \left| 
(e_k^{(0,2)},  \chi)_{L^2(0,2)} (e_k^{(0,2)}, \psi' f_1)_{L^2(0,2)}\right| = \Tr(|\chi \otimes \psi' f|),
 \]
 where $\chi \otimes \psi' f_1$ stands for the operator on $L^2(0,2)$ with smooth kernel $\chi \otimes \psi' f_1$.
 So the sum over $k$ on the right hand side of \eqref{eq Trace exp 2} indeed converges, and the claim follows with $B:= B_0 \Tr(|\chi \otimes \psi' f|)$.
 %
\end{proof}

\begin{lemma}\label{lem At trace}
If $A_{g}^t(D_C, \psi')$ converges, then $\varphi_2 e_P^{-t D_C^- D_C^+} (D_C^+)^{-1} \psi' f_1$ is $g$-trace class, and its $g$-trace equals $-A_{g}^t(D_C, \psi')$.
\end{lemma}
\begin{proof}
Let $\chi \in C^{\infty}_c(0,\infty)$ be equal to $1$ on $\supp(\psi')$. Then the all $s>0$, the operator $\chi e_P^{-s D_C^- D_C^+} (D_C^+)^{-1} \psi' f_1$ has a compactly supported smooth kernel, is trace class, and
\[
A_{g}^t(D_C, \psi') =  \int_t^{\infty} \Tr(\chi g e_P^{-s D_C^- D_C^+} D_C^- \psi' f_1)\, ds.
\]
Let $\{e_j\}_{j=1}^{\infty}$ be a Hilbert basis of $L^2(S)$. Then the right hand side equals
\beq{eq At g trace}
\int_t^{\infty}\sum_{j=1}^{\infty} \left( e_j, \chi g e_P^{-s D_C^- D_C^+} D_C^-\psi' f_1 e_j\right)_{L^2(S)}\, ds.
\eeq
This combined integral and sum converges absolutely by Lemma \ref{lem Trace exp} and the fact that
\[
|\chi g e_P^{-s D_C^- D_C^+} D_C^- \psi' f_1| = |\chi e_P^{-s D_C^- D_C^+} D_C^-\psi' f_1|,
\]
because $\chi$ is $G$-invariant, and $g$ acts unitarily.
So \eqref{eq At g trace} equals
\[
\sum_{j=1}^{\infty}
\int_t^{\infty} \left( e_j, \chi g e_P^{-s D_C^- D_C^+} D_C^{-} \psi' f_1 e_j\right)_{L^2(S)}\, ds.
\]
Using the spectral decomposition of \eqref{eq DC-DC+}, we find that this equals
\[
\begin{split}
-
\sum_{j=1}^{\infty}
\left( e_j, \chi g e_P^{-t D_C^- D_C^+} (D_C^+)^{-1} \psi' f_1 e_j\right)_{L^2(S)} &= -\Tr(\chi g e_P^{-t D_C^- D_C^+} (D_C^+)^{-1} \psi' f_1) \\
&=
-\Tr_g(\varphi_2 e_P^{-t D_C^- D_C^+} (D_C^+)^{-1} \psi' f_1).
\end{split}
\]

\end{proof}

\subsection{An  index formula with a smooth cutoff function}

\begin{lemma}\label{lem S+' S-}
The operators $S_+'$ and $S_-$ in \eqref{eq Spm} are $g$-trace class, and
\beq{eq lem 5.3}
\Tr_g(S_+') - \Tr_g(S_-) = \Tr(g \circ e^{-t\tilde D^- \tilde D^+}\psi) -  \Tr(g \circ e^{-t\tilde D^+ \tilde D^-}\psi).
\eeq
\end{lemma}
\begin{proof}
The  proof of Lemma 5.3 in \cite{HWW} shows that $S_+'$ and $S_-$ are $g$-trace class, and that the left hand side of \eqref{eq lem 5.3} equals
\[
\Tr_g(e^{-t\tilde D^- \tilde D^+}\psi) -  \Tr_g(e^{-t\tilde D^+ \tilde D^-}\psi).
\]
 One then uses the fact that the heat operator $e^{-t\tilde D^2}$ is trace class, and the comment below Definition \ref{def g trace} on $g$-traces of trace-class operators.
\end{proof}

%
%
%

\begin{proposition}\label{prop general cusp contr}
If the number $A_{g}^t(D_C, \psi')$ in \eqref{eq def At}
converges
for all $t \in (0,1]$, then
the operator $S_+$ is $g$-trace class. And
\beq{eq general cusp contr}
\Tr_g(S_+) - \Tr_g(S_+')=A_{g}^t(D_C, \psi') + F(t),
\eeq
for a function $F$ satisfying $\lim_{t\downarrow 0}F(t) = 0$.
\end{proposition}
\begin{proof}
Lemma \ref{lem Sj} implies that
\begin{multline}\label{eq S+ S+'}
S_+ - S_+' = (\varphi_1 \tilde S_+ \psi - \psi \tilde S_+ \varphi_1) -
(\psi \tilde Q \sigma \varphi_1' f_1 + (1-\psi) (D_C^+)^{-1} \sigma \varphi_2' f_1) \\
+
\bigl(\varphi_1 \tilde Q\sigma \psi' f_1 - \varphi_2 (D_C^+)^{-1} \sigma \psi' f_1\bigr) .
\end{multline}

The first term on the right hand side of \eqref{eq S+ S+'} is $g$-trace class because it is a smooth kernel operator on the compact manifold $\tilde M$. And its $g$-trace is zero  by the trace property of $\Tr_g$, Lemma \ref{lem Trg trace}.

The second term on the right hand side of \eqref{eq S+ S+'} is $g$-trace class with $g$-trace zero, because it is a smooth kernel operator that vanishes on the set $\{ (g^{-1}m,m); m \in U\}$.
Here we use the facts that $\tilde Q$ and $(D_C^+)^{-1}$ are pseudo-differential operators, and that the supports of $\psi$ and $\varphi_1'$ are $G$-invariant and disjoint, and the same is true for the supports of $(1-\psi)$ and $\varphi_2'$ .
See also the proof of Lemma 5.3 in \cite{HWW}.

The third term on the right hand side of \eqref{eq S+ S+'} equals
\beq{eq general cusp contr 2}
(\varphi_1 \tilde Q - \varphi_2 Q_C) \sigma \psi' f_1 -  \varphi_2 ( (D_C^+)^{-1} -Q_C) \sigma \psi' f_1.
\eeq
Proposition \ref{prop lem 5.5} implies that the first of these terms is $g$-trace class for all $t>0$, and that its $g$-trace, which will be $F(t)$ in \eqref{eq general cusp contr},  goes to zero in the limit $t \downarrow 0$.
The second term in \eqref{eq general cusp contr 2} equals
\[
-\varphi_2  e_P^{-t D_C^- D_C^+} (D_C^+)^{-1}  \sigma \psi' f_1.
\]
Hence the claim follows by Lemma \ref{lem At trace}.
%
\end{proof}

We obtain a version of Theorem \ref{thm index general} involving the function $\psi$.

\begin{theorem}\label{thm index general psi}
The limit $\lim_{t\downarrow 0}A_{g}^t(D_C, \psi')$ converges, and
\beq{eq index general psi}
\ind_G(D)(g) = \int_{M^g} \psi|_{M^g} \AS_g(D)+\lim_{t\downarrow 0}A_{g}^t(D_C, \psi').
\eeq
\end{theorem}
\begin{proof}
Convergence of  $A_{g}^t(D_C, \psi')$ for $t>0$ is Proposition \ref{prop conv cusp contr}.

The operator $S_-$ is $g$-trace class by Lemma \ref{lem S+' S-}. And the operator $S_+$ is $g$-trace class if $A_{g}^t(D_C, \psi')$ converges, by Proposition \ref{prop general cusp contr}. So then $D^+$ is $g$-Fredholm. And because it is also Fredholm in the usual sense, because of \eqref{eq D inv infty},  Lemma \ref{lem inde ind} implies that
\[
\ind_G(D)(g) = \ind_g(D) = \Tr_g(S_+) - \Tr_g(S_-).
\]
The operator $S_+'$ is $g$-trace class by Lemma \ref{lem S+' S-}, so the right hand side equals
\[
\bigl( \Tr_g(S_+) - \Tr_g(S_+') \bigr) + \bigl(\Tr_g(S_+')  - \Tr_g(S_-) \bigr).
\]
By Lemma \ref{lem S+' S-} and Proposition \ref{prop general cusp contr}, this equals
\beq{eq index general 3}
\Tr(g\circ e^{-t\tilde D^- \tilde D^+}\psi) -  \Tr(g\circ e^{-t\tilde D^+ \tilde D^-}\psi) + A_{g}^t(D_C, \psi') +F(t),
\eeq
where $\lim_{t\downarrow 0}F(t) = 0$. This total expression is independent of $t$, and by standard heat kernel asymptotics, the first two terms converge to the first term on the right of \eqref{eq index general psi} as $t \downarrow 0$. (See e.g.\ the proof of Theorem 6.16 in \cite{BGV}.) Hence $\lim_{t\downarrow 0} A_{g}^t(D_C, \psi')$ converges, and \eqref{eq index general psi} holds.
\end{proof}

\subsection{The limit $\psi' \to -\delta_{a'}$}

Theorem \ref{thm index general psi} implies Theorem \ref{thm index general} via a limit where $\psi'$ approaches minus the Dirac delta distribution $\delta_{a'}$ at $a' \in (0,1)$. We make this precise in this subsection.
For $a' \in (0,1)$ and $t>0$, we write
\[
A_g^t(D_C, a') := - f_1(a')\int_t^{\infty} \int_N \tr(g \lambda_s^P(g^{-1}n, a'; n, a'))\, dn\, ds,
\]
when this converges.
\begin{lemma}\label{lem psi delta}
For all  $a' \in (0,1)$ and $t>0$, $A_g^t(D_C, a')$ converges. And for all $t>0$ and $\varepsilon>0$,  there is a $\psi$ as at the start of Subsection \ref{sec parametrices} such that
\beq{eq est At delta}
\left| A_g^t(D_C, \psi') -  A_g^t(D_C, a')\right| < \varepsilon.
\eeq
\end{lemma}
\begin{proof}
For $s>0$ and $x\in  (0,1)$, we write
\[
F(s, x) := \int_N \tr(g \lambda_s^P(g^{-1}n, x; n, x)f_1(x))\, dn.
\]
This defines a smooth function on $(0, \infty) \times (0,1)$. To avoid complications near the boundaries, we consider functions $\chi_1 \in C^{\infty}(0, \infty)$ and $\chi_2 \in C^{\infty}_c(0,1)$  with values in $[0, 1]$, such that
\[
\begin{split}
\chi_1(s) &= 0 \quad  \text {for $s \leq t/2$};\\
\chi_1(s) &= 1\quad \text {for $s \geq t$};\\
\chi_2(x) &= 1\quad \text {for $x  \in [a'/2, a'']$},
\end{split}
\]
for some $a'' \in (a'/2, 1)$.
Then by Lemma \ref{lem est kappa F}, and the expression for $\lambda_s^P$ in terms of Dirichlet heat kernels from Proposition \ref{prop eF eP},
 the extension of the function $(\chi_1 \otimes \chi_2)f$  by zero outside $(0, \infty) \times (0,1)$ lies in the Schwartz space $\calS(\R^2)$.

Fix $\varepsilon>0$.
Now the distribution $\delta_{\R \times \{a'\}} \in \calS'(\R^2)$ defined by integrating functions over $\R \times \{a'\}$, can be approximated arbitrarily closely by functions $1 \otimes (-\psi')$ in the sense of distributions, with $\psi$ as at the start of Subsection \ref{sec parametrices}, and with $\psi'$ supported in $(a'/2,a'')$. So we can choose $\psi$ with these properties such that
\beq{eq psi delta}
\left|  \int_{\R^2} \chi_1(s)  F(s, x) (-\psi'(x)) \, ds\, dx - \int_{\R} \chi_1(s) F(s, a')  \, ds\, \right|< \varepsilon/2.
\eeq
Here we use the fact that $\chi_2(x) = 1$  for $x \in \supp(\psi')$, so the first  integral does not change if we insert $\chi_2(x)$  into the integrand.

By choosing $\chi_1$ supported close enough to $[t, \infty)$, we can ensure that the left hand side of \eqref{eq psi delta} is within $\varepsilon/2$ of the left hand side of \eqref{eq est At delta}.
\end{proof}

\begin{proof}[Proof of Theorem \ref{thm index general}.]
To prove that  $A_{g}(D_C, a')$ converges, our starting point is the fact that $\ind_G(D)(g)$ equals \eqref{eq index general 3}, as shown in the proof of Theorem \ref{thm index general psi}.
Let $\varepsilon > 0$. By standard heat kernel asymptotics, we can choose $t>0$ such that
\[
\left|  \Tr(g\circ e^{-t\tilde D^- \tilde D^+}\psi) -  \Tr(g\circ e^{-t\tilde D^+ \tilde D^-}\psi) - \int_{M^g} \psi|_{M^g} \AS_g(D)\right| < \varepsilon.
\]
Fix $t>0$ such that this holds, and also $|F(t)|< \varepsilon$.
For
 $a' \in(0,1)$,
choose $\psi$ as in Lemma \ref{lem psi delta}. By choosing $\psi$ so that $\psi'$ is close enough to $\delta_{N \times \{a'\}}$, we can ensure that also
\[
\left| \int_{M^g} \psi|_{M^g} \AS_g(D) - \int_{Z^g \cup (N^g \times (a,a'])}  \AS_g(D) \ \right|< \varepsilon.
\]
Then we find that
\[
\left| \ind_G(D)(g) -  \int_{Z^g \cup (N^g \times (a,a'])}  \AS_g(D) - A_g^t(D_C,  a') \right|< 4\varepsilon.
\]
So $A_g(D_C, a')  = \lim_{t \downarrow 0}A_g^t(D_C, a') $ converges, and \eqref{eq index general} holds.
\end{proof}

\section{The Donnelly--APS index theorem} \label{sec APS}

In this section, we use Corollary \ref{cor index general} to prove Corollary \ref{cor APS}.
 The proof is based on an explicit expression for the APS heat operator $e_P^{-sD_C^- D_C^+}$, given in \cite{APS1}. In Subsection 5.1 of \cite{HW21b}, we give a spectral version of the geometric computation in this section.

\subsection{Vanishing of an integral}

In the proof of Corollary \ref{cor APS}, we will use the following vanishing result.
\begin{proposition}\label{prop vanish}
Let $(\lambda_j)_{j=1}^{\infty}$ and $(a_j)_{j=1}^{\infty}$ be  sequences in $\R$ such that $|\lambda_1|>0$, and $|\lambda_j| \leq |\lambda_{j+1}|$ for all $j$, and such that
there are $c_1,c_2, c_3, c_4>0$ such that for all $j$,
\beq{eq growth lambda}
\begin{split}
|\lambda_j| &\geq c_1 j^{c_2};\\
|a_j| & \leq c_3 j^{c_4}.
\end{split}
\eeq
Then for all $a'>0$,
\beq{eq vanish}
\int_0^{\infty} \sum_{j=1}^{\infty} \sgn(\lambda_j)  a_j \frac{e^{-\lambda_j^2 s} e^{-a'^2/s}}{\sqrt{s}} \left(\frac{a'}{s} - |\lambda_j| \right) \, ds = 0.
\eeq
\end{proposition}
The main part of the proof is to show that the order of the integral and sum may be interchanged. The rest of the argument is a substitution.

We use the notation of Proposition \ref{prop vanish}. We first consider a version of  \eqref{eq vanish} where the integral over $s$ starts at a positive number $t$, which we fix from now on.
\begin{lemma}\label{lem abs conv t}
The integral and sum
\beq{eq vanish t}
\int_t^{\infty}  \sum_{j=1}^{\infty}  \sgn(\lambda_j) a_j \frac{e^{-\lambda_j^2 s} e^{-a'^2/s}}{\sqrt{s}}\left(\frac{a'}{s} - |\lambda_j| \right) \, ds
\eeq
converge absolutely.
\end{lemma}
\begin{proof}
For all $s\geq t$, 
\[
\left|a_j
\frac{e^{-\lambda_j^2 s} e^{-a'^2/s}}{\sqrt{s}} \left(\frac{a'}{s} - |\lambda_j| \right) \right| \leq |a_j|\frac{e^{-\lambda_j^2 t/2} e^{-\lambda_1^2 s/2}}{\sqrt{t}} \left( \frac{a'}{t} + |\lambda_j| \right).
\]
The sum
\[
\sum_{j=1}^{\infty} |a_j| e^{-\lambda_j^2 t/2} \left( \frac{a'}{t} + |\lambda_j| \right)
\]
converges by \eqref{eq growth lambda}. The integral
\[
\int_t^{\infty} e^{-\lambda_1^2 s/2}\, ds
\]
converges because $|\lambda_1|>0$.
\end{proof}

\begin{lemma}\label{lem subst}
For all nonzero $\lambda \in \R$, 
\[
\int_t^{\infty} \frac{e^{-\lambda^2 s} e^{-a'^2/s}}{\sqrt{s}} \frac{a'}{s}  \, ds =
 |\lambda|
\int_0^{a'^2/\lambda^2 t} \frac{e^{-\lambda^2 s} e^{-a'^2/s}}{\sqrt{s}} \, ds.
\]
\end{lemma}
\begin{proof}
Substitute $s \mapsto \frac{a'^2}{\lambda^2 s}$.
\end{proof}

By Lemmas \ref{lem abs conv t} and \ref{lem subst} and the Fubini--Tonelli theorem, the expression \eqref{eq vanish t} equals
\beq{eq vanish 2}
\sum_{j=1}^{\infty} \lambda_j a_j
\left(
\int_0^{a'^2/\lambda_j^2t}
 \frac{e^{-\lambda_j^2 s} e^{-a'^2/s}}{\sqrt{s}}\, ds -
\int_t^{\infty}
 \frac{e^{-\lambda_j^2 s} e^{-a'^2/s}}{\sqrt{s}}\, ds
 \right).
\eeq
We will bound the term for each $j$ by a summable function independent of $t$, so that the limit $t \downarrow 0$ may be taken inside the sum by dominated convergence. We consider two cases separately: $t \leq a'/|\lambda_j|$ and $t \geq a'/|\lambda_j|$.

\begin{lemma}\label{lem small t est 1}
For all $t> 0$ and nonzero $\lambda \in \R$ with $t \leq a'/|\lambda|$,
\[
\int_0^t e^{-2\lambda^2 s - a'^2/s}\, ds \leq \frac{a'}{|\lambda|} e^{-a'|\lambda|}.
\]
\end{lemma}
\begin{proof}
If $s \leq t \leq a'/|\lambda|$, then $e^{-a'^2/s} \leq e^{-a'|\lambda|}$. So the left hand side is at most equal to $t e^{-a'|\lambda|}$, which is at most equal to the right hand side.
\end{proof}

\begin{lemma}\label{lem small t subst}
For all $t> 0$ and nonzero $\lambda \in \R$ with $t \leq a'/|\lambda|$,
\beq{eq small t subst}
\int_0^{a'^2/\lambda^2t}
 \frac{e^{-\lambda^2 s} e^{-a'^2/s}}{\sqrt{s}}\, ds -
\int_t^{\infty}
 \frac{e^{-\lambda^2 s} e^{-a'^2/s}}{\sqrt{s}}\, ds
 = \int_0^t
 \frac{e^{-\lambda^2 s} e^{-a'^2/s}}{\sqrt{s}} \left(1-\frac{a'}{|\lambda| s} \right)\, ds.
\eeq
\end{lemma}
\begin{proof}
If $t \leq a'/|\lambda|$, then $t \leq a'^2/\lambda^2t$.
So the left hand side of \eqref{eq small t subst} equals
\[
\int_0^{t}
 \frac{e^{-\lambda^2 s} e^{-a'^2/s}}{\sqrt{s}}\, ds -
\int_{a'^2/\lambda^2 t}^{\infty}
 \frac{e^{-\lambda^2 s} e^{-a'^2/s}}{\sqrt{s}}\, ds.
\]
By a substitution $s\mapsto a'^2/\lambda^2 s$ (or by Lemma \ref{lem subst}),
\[
\int_{a'^2/\lambda^2 t}^{\infty}
 \frac{e^{-\lambda^2 s} e^{-a'^2/s}}{\sqrt{s}}\, ds = \frac{a'}{|\lambda|} \int_0^t
  \frac{e^{-\lambda^2 s} e^{-a'^2/s}}{\sqrt{s}} \frac{1}{s}\, ds.
\]
\end{proof}

\begin{lemma}\label{lem est small t 3}
For all $t> 0$ and nonzero $\lambda \in \R$ with $t \leq a'/|\lambda|$,
\beq{eq est small t}
\left| \int_0^{a'^2/\lambda^2t}
 \frac{e^{-\lambda^2 s} e^{-a'^2/s}}{\sqrt{s}}\, ds -
\int_t^{\infty}
 \frac{e^{-\lambda^2 s} e^{-a'^2/s}}{\sqrt{s}}\, ds  \right|
 \leq
f_1(t,|\lambda|) e^{-a'|\lambda|/2},
\eeq
for a function $f_1$ that is bounded in $t \in (0,1]$ and in $|\lambda|$ in a subset of $\R$ with a positive lower bound.
\end{lemma}
\begin{proof}
By Lemma \ref{lem small t subst} and the Cauchy--Schwartz inequality, the left hand side of  \eqref{eq est small t} is at most equal to
\[
\left( \int_0^t e^{-2\lambda^2 s}  e^{-a'^2 /s}\, ds\right)^{1/2}
\left( \int_0^t  \frac{e^{-a'^2/s}}{s}
 \left(1-\frac{a'}{|\lambda| s} \right)^2
 \, ds\right)^{1/2}.
\]
The second factor is  bounded in $t \in (0,1]$ and in $|\lambda|$ in a set with a positive lower bound. By Lemma \ref{lem small t est 1}, the first factor is at most equal to
\[
\left( \frac{a'}{|\lambda|} e^{-a'|\lambda|} \right)^{1/2}.
\]
\end{proof}

Next, we turn to the case $t \geq a'/|\lambda_j|$.
\begin{lemma}\label{lem est large t 3}
For all nonzero $\lambda \in \R$ and $t \geq a'/|\lambda|$,
\beq{eq est large t}
\left| \int_0^{a'^2/\lambda^2 t}
 \frac{e^{-\lambda^2 s} e^{-a'^2/s}}{\sqrt{s}}\, ds
\right| \leq f_2(|\lambda|) e^{-a'|\lambda|/2},
\eeq
for a function $f_2$ that is bounded in $|\lambda|$ in a subset of $\R$ with a positive lower bound.
\end{lemma}
\begin{proof}
If $s \leq a'^2/\lambda^2 t \leq a'/|\lambda|$, then $e^{-a'^2/2s} \leq e^{-a'|\lambda|/2}$. This implies that the left hand side of \eqref{eq est large t} is at most equal to
\beq{eq large t est 1}
\frac{a'^2}{\lambda^2 t} e^{-a'|\lambda|/2} \max_{0 \leq s \leq a'^2/\lambda^2 t} \frac{e^{-a'^2/2s}}{\sqrt{s}}.
\eeq
The condition on $t$ implies that $\frac{a'^2}{\lambda^2 t} \leq a'/|\lambda|$, so the claim follows.
\end{proof}

\begin{lemma}\label{lem est large t 4}
For all  nonzero $\lambda \in \R$ and $t \geq a'/|\lambda|$,
\beq{eq est large t 2}
\left| \int_t^{\infty}
 \frac{e^{-\lambda^2 s} e^{-a'^2/s}}{\sqrt{s}} \, ds \right| \leq f_3(|\lambda|) e^{-a'|\lambda|/2},
\eeq
for a function $f_3$ that is bounded in $|\lambda|$ in a subset of $\R$ with a positive lower bound.
\end{lemma}
\begin{proof}
If $s \geq t \geq a'/|\lambda|$, then
\[
 \frac{e^{-\lambda^2 s} e^{-a'^2/s}}{\sqrt{s}}  \leq e^{-a'|\lambda|/2}  \frac{e^{-\lambda^2 s/2} e^{-a'^2/s}}{\sqrt{s}}.
\]
So the left hand side of \eqref{eq est large t 2} is at most equal to
\[
e^{-a'|\lambda|/2} \int_0^{\infty} \frac{e^{-\lambda^2 s/2} e^{-a'^2/s}}{\sqrt{s}}\, ds.
\]
And
\[
\int_0^{\infty} \frac{e^{-\lambda^2 s/2} e^{-a'^2/s}}{\sqrt{s}}\, ds \leq
\int_0^{1} \frac{ e^{-a'^2/s}}{\sqrt{s}}\, ds+
\int_1^{\infty} e^{-\lambda^2 s/2} \, ds.
\]
The first term on the right is constant in $\lambda$, the second is bounded in $\lambda$ in sets with positive lower bounds.
\end{proof}

We now combine the cases $t \leq a'/|\lambda_j|$ and $t \geq a'/|\lambda_j|$.
\begin{lemma}\label{lem dominator}
For all $t> 0$ and nonzero $\lambda \in \R$,
\[
\left| \int_0^{a'^2/\lambda^2t}
 \frac{e^{-\lambda^2 s} e^{-a'^2/s}}{\sqrt{s}}\, ds -
\int_t^{\infty}
 \frac{e^{-\lambda^2 s} e^{-a'^2/s}}{\sqrt{s}}\, ds  \right|
 \leq f(t, |\lambda|) e^{-a'|\lambda|/2},
\]
 for a function $f$ that is bounded in $t \in (0,1]$ and in $|\lambda|$ in a subset of $\R$ with a positive lower bound.
\end{lemma}
\begin{proof}
For $t>0$ and nonzero $\lambda \in \R$, let
\[
f(t, |\lambda|):= \max\bigl(f_1(t, |\lambda|), f_2(|\lambda|)+f_3(|\lambda|) \bigr),
\]
for $f_1$ as in Lemma \ref{lem est small t 3}, $f_2$ as in Lemma \ref{lem est large t 3} and $f_3$ as in Lemma \ref{lem est large t 4}. Then  the
claim follows from  Lemma \ref{lem est small t 3} if $t \leq a'/|\lambda|$ and from Lemmas   \ref{lem est large t 3} and \ref{lem est large t 4} if $t\geq a'/|\lambda|$.
\end{proof}

\begin{proof}[Proof of Proposition \ref{prop vanish}.]
By Lemmas \ref{lem abs conv t} and \ref{lem subst}, the left hand side of \eqref{eq vanish} equals
\beq{eq vanish 3}
\lim_{t \downarrow 0}
\sum_{j=1}^{\infty} \lambda_j  a_j
\left(
\int_0^{a'^2/\lambda_j^2t}
 \frac{e^{-\lambda_j^2 s} e^{-a'^2/s}}{\sqrt{s}}\, ds -
\int_t^{\infty}
 \frac{e^{-\lambda_j^2 s} e^{-a'^2/s}}{\sqrt{s}}\, ds
 \right).
\eeq
 For every $j$ separately,
\[
\lim_{t \downarrow 0}
\left(
\int_0^{a'^2/\lambda_j^2t}
 \frac{e^{-\lambda_j^2 s} e^{-a'^2/s}}{\sqrt{s}}\, ds -
\int_t^{\infty}
 \frac{e^{-\lambda_j^2 s} e^{-a'^2/s}}{\sqrt{s}}\, ds \right)= 0.
\]

By Lemma \ref{lem dominator}, the term in \eqref{eq vanish 3} for each $j$ is bounded by a constant times $|\lambda_j a_j|e^{-a' |\lambda_j|/2}$ for all $t \in (0,1]$. The sum over $j$ of this bounding expression converges by \eqref{eq growth lambda}. So by dominated convergence, \eqref{eq vanish 3} equals zero.
\end{proof}

\subsection{Computation of the contribution from infinity}


For every $\lambda \in \spec(D_N^+)$, let $\{\varphi_{\lambda}^1, \ldots, \varphi_{\lambda}^{m_{\lambda}}\}$ be an orthonormal basis of $\ker(D_N^+ - \lambda)$. We identify $S^+|_N \cong (S^+|_N)^*$ using the metric, so that for all $n,n' \in N$, $\lambda \in \spec(D_N^+)$ and $j \in \{1,\ldots, m_{\lambda}\}$,
\[
  \varphi_{\lambda}^j(n) \otimes \varphi_{\lambda}^j(n') \in  S^+_n \otimes (S^+_{n'})^* = \Hom(S^+_{n'}, S^+_n).
\]

In this subsection, we take the subset $J \subset \spec(D_N^+)$ in the definition of $P$ in Subsection \ref{sec index thm general} to be the positive part of the spectrum of $D_N^+$. Then, as in \cite{APS1}, $P$ is orthogonal projection onto the positive eigenspaces of $D_N^+$.
\begin{proposition}\label{prop kappa cyl}
The Schwartz kernel $\lambda_s^P$ of $e_P^{-sD_C^-D_C^+}D_C^-$ is given by
\begin{multline} \label{eq tilde kappa s cyl}
\lambda_s^P(n, y; n', y') = \\
%
\sum_{\lambda > 0} \sum_{j=1}^{m_{\lambda}} \frac{e^{-\lambda^2 t}}{\sqrt{4\pi t}}\left[
e^{-\frac{(y-y')^2}{4t}}\left( \frac{y-y'}{2t} +\lambda \right)
+
e^{-\frac{(y+y')^2}{4t}} \left( \frac{y+y'}{2t} - \lambda \right)
 \right]
 \varphi_{\lambda}^j(n) \otimes \varphi_{\lambda}^j(n') \\
+\sum_{\lambda<0}\sum_{j=1}^{m_{\lambda}} \frac{e^{-\lambda^2 t}}{\sqrt{4\pi t}}
\biggl[
 e^{-\frac{(y-y')^2}{4t}} \left(\frac{y-y'}{2t} +\lambda \right)
 + e^{-\frac{(y+y')^2}{4t}}\left( - \frac{y+y'}{2t} + \lambda \right)  \biggr] \varphi_{\lambda}^j(n) \otimes \varphi_{\lambda}^j(n') \\
 -\sum_{\lambda<0}\sum_{j=1}^{m_{\lambda}}  \frac{\lambda}{\sqrt{\pi t}}
  e^{-\lambda(y+y')}e^{-\left(\frac{y+y'}{2\sqrt{t}}-\lambda\sqrt{t}\right)^2}
 \varphi_{\lambda}^j(n) \otimes \varphi_{\lambda}^j(n'),
\end{multline}
for all $n,n' \in N$ and $y,y' \in (0, \infty)$. Here the outer sums run over the positive and negative eigenvalues $\lambda$ of $D_N^+$, respectively.
\end{proposition}
\begin{proof}
As before, let $\kappa_s^{P, +}$
be the Schwartz kernel of $e_P^{-sD_C^-D_C^+}$.
%
From the explicit heat kernel expression in (2.16) and (2.17) in \cite{APS1}, we have for all $n,n' \in N$ and $y,y' \in (0, \infty)$,
\begin{align*}
 &\kappa_s^{P, +}(n,y;n',y')=\sum_{\lambda >0} \sum_{j=1}^{m_{\lambda}}\frac{e^{-\lambda^2 t}}{\sqrt{4\pi t}}\left[ e^{-\frac{(y-y')^2}{4t}} - e^{-\frac{(y+y')^2}{4t}}\right]
 \varphi_{\lambda}^j(n) \otimes \varphi_{\lambda}^j(n')
\\
& +\sum_{\lambda<0}  \sum_{j=1}^{m_{\lambda}} \left\{
\frac{e^{-\lambda^2 t}}{\sqrt{4\pi t}}\left[ e^{-\frac{(y-y')^2}{4t}} + e^{-\frac{(y+y')^2}{4t}}\right]+\lambda e^{-\lambda(y+y')}\erfc\left(\frac{y+y'}{2\sqrt{t}}-\lambda\sqrt{t}\right)
\right\}
 \varphi_{\lambda}^j(n) \otimes \varphi_{\lambda}^j(n'),
\end{align*}
where  $\erfc(x)=\frac{2}{\sqrt{\pi}}\int_{x}^{\infty}e^{-\xi^2}d\xi$ is the complementary error function.
In this case,
\[
D_C^{\pm} = \pm \frac{d}{dy} + D_N^+.
\]
We can find   $\lambda_s^P$  by applying $(D_C^-)^* = D_C^+$ to the second variable in  $\kappa_s^{P, +}$. In other words,
\[
\lambda_s^P= \left(1_{S_C^+} \otimes \left(\frac{\partial}{\partial y'} + D_N^+ \right) \right) \kappa_s^{P, +},
\]
where, as above, we view $\kappa_s^{P, +}$ as a section of $S_C^+\boxtimes S_C^+$.
This implies \eqref{eq tilde kappa s cyl} by a direct computation.
%
%
\end{proof}

\begin{proposition}\label{prop eta g}
In the situation of Corollary \ref{cor APS},
we have for all $a'>0$,
\[
A_{g}(D_C, a') = -\frac{1}{2} \eta_g(D_N^+).
\]
\end{proposition}
\begin{proof}
The function $f_1$ in \eqref{eq def At} now equals $1$. We use Proposition \ref{prop kappa cyl} and the equality
\[
\sum_{j=1}^{m_{\lambda}}
\int_{N} \tr(g\varphi_{\lambda}^j(g^{-1}n) \otimes \varphi_{\lambda}^j(n)) \, dn =
\sum_{j=1}^{m_{\lambda}} (g\cdot \varphi_{\lambda}^j, \varphi_{\lambda}^j)_{L^2} = \tr(g|_{\ker(D_N^+ - \lambda)})
\]
to compute
\begin{multline*}
A_g(D_C, a')
=-\int_0^{\infty}\int_N
\tr(g\lambda_s^P(g^{-1}n, a'; n, a')) dn\, ds\\
=-\int_0^{\infty}\sum_{\lambda > 0}\frac{e^{-\lambda^2 s}}{\sqrt{4\pi s}}\tr(g|_{\ker(D_N^+ - \lambda)})\left[e^{-\frac{a'^2}{s}}\frac{a'}{s}+\lambda-\lambda e^{-\frac{a'^2}{s}}\right]\,ds \\
-\int_0^{\infty}\sum_{\lambda<0}\frac{e^{-\lambda^2 s}}{\sqrt{4\pi s}}\tr(g|_{\ker(D_N^+ - \lambda)})\left[-e^{-\frac{a'^2}{s}}\frac{a'}{s}+\lambda+\lambda
e^{-\frac{a'^2}{s}}\right] \, ds \\
+\int_0^{\infty}\sum_{\lambda<0}\tr(g|_{\ker(D_N^+ - \lambda)})\frac{\lambda}{\sqrt{\pi s}}e^{-2\lambda a'}e^{-\left(\frac{a'}{\sqrt{s}}-\lambda\sqrt{s}\right)^2}\, ds.
\end{multline*}
The last two lines can be simplified to
\[
-\int_0^{\infty}\sum_{\lambda<0}\frac{e^{-\lambda^2 s}}{\sqrt{4\pi s}}\tr(g|_{\ker(D_N^+ - \lambda)})\left[-e^{-\frac{a'^2}{s}}\frac{a'}{s}+\lambda-\lambda e^{-\frac{a'^2}{s}}\right]\, ds.
\]
We find that
%
\begin{multline} \label{eq Ag cyl 2}
A_{g}(D_C, a')=-\int_0^{\infty}\sum_{\lambda\in \spec(D_N^+)}
\tr(g|_{\ker(D_N^+ - \lambda)})
\frac{e^{-\lambda^2 s}}{\sqrt{4\pi s}}\lambda \, ds\\
-\int_{0}^{\infty}\sum_{\lambda \in \spec(D_N^+)} \sgn(\lambda)\tr(g|_{\ker(D_N^+ - \lambda)})\frac{e^{-\lambda^2 s}e^{-\frac{a'^2}{s}}}{\sqrt{4\pi s}}\left(\frac{a'}{s}-|\lambda|\right)\, ds.
\end{multline}
We apply Proposition \ref{prop vanish} to the last line, with $\lambda_j$ the $j$th eigenvalue of $D_N^+$ (ordered by absolute values), and $a_j = \tr(g|_{\ker(D_N^+ - \lambda_j)})$. Then the condition on $\lambda_j$ in Proposition \ref{prop vanish} holds by Weyl's law, and the condition on $a_j$ holds because
\[
|a_j| \leq  \dim(\ker(D_N^+ - \lambda_j)),
\]
which grows at most polynomially by Weyl's law. Hence Proposition \ref{prop vanish}  implies that the last line of \eqref{eq Ag cyl 2} equals zero.
 We conclude that
\[
A_{g}(D_C, a')
=-\frac{1}{2\sqrt{\pi}}\int_0^{\infty}\frac{1}{\sqrt{ s}}\Tr(g e^{-s(D_N^+)^2}D_N^+)\, ds
=-\frac{1}{2}\eta_g(D_N^+).
\]
\end{proof}

\begin{remark}\label{rem cyl zero at a}
It follows from Proposition \ref{prop kappa cyl} that for all $n,n' \in N$,
\[
\lambda_s^P(n, 0; n', 0)  = 0.
\]
This is consistent with Remark \ref{rem zero at a}. This also shows that it is important in Proposition \ref{prop eta g} that $a'>0$.
\end{remark}

\begin{proof}[Proof of Corollary \ref{cor APS}]
In the situation of Corollary \ref{cor APS},  $D$ satisfies the conditions of Theorem \ref{thm index general} and Corollary \ref{cor index general}.
In particular, see Propositions 2.5 and 2.12 in \cite{APS1} for invertibility of \eqref{eq DC+} and \eqref{eq DC-} and self-adjointness of \eqref{eq DC-DC+} and \eqref{eq DC+DC-}.
See also Theorem X.25 in \cite{RSII}.
 (Alternatively, one can  use generalisations of these results from \cite{APS1},   Propositions 3.12 and 3.15 in \cite{HW21b}.)
Because $\ind_G(D)(g) = \ind_G^{\APS}(D_Z)(g)$,  Corollary \ref{cor index general} and Proposition \ref{prop eta g} imply the claim.
\end{proof}

\begin{remark}\label{rem APS bdry cond}
The proof of Proposition \ref{prop eta g} can be simplified somewhat if one uses
Proposition \ref{prop eF eP}
 to replace APS heat operators by Dirichlet heat operators and spectral projections. 
 We have chosen to give a more direct proof, to illustrate the role of APS boundary conditions.

We also see in this example why it is important that the heat operator in \eqref{eq def At} is defined with APS-type boundary conditions at $N$. Indeed, consider the  expression $A_g^F(D_C, a')$ analogous to \eqref{eq def At}, with the heat operator $e_P^{-sD_C^- D_C^+}$  replaced by the Dirichlet heat operator $e_F^{-sD_C^- D_C^+}$. By a similar computation to the one in this subsection, one finds that, in the current setting,  $A_g^F(D_C, a')$ equals a version of \eqref{eq Ag cyl 2} where the factor $\frac{a'}{s} - |\lambda|$ is replaced by $\sgn(\lambda)\frac{a'}{s} - |\lambda|$. Because of this difference, Proposition \ref{prop vanish} does not apply, and  $A_g^F(D_C, a')$ is generally not equal to the delocalised $\eta$-invariant in this case. This implies that a version of Theorem \ref{thm index general} with  $A_{g}(D_C, a')$ replaced by  $A_g^F(D_C, a')$ is not true in general.
%
\end{remark}

\bibliographystyle{plain}

\bibliography{mybib}

\end{document}